\documentclass[11pt]{amsart}
\usepackage{amsmath}
\usepackage{amssymb}
\usepackage{amscd}
\usepackage{url}
\usepackage[all]{xy} 

\usepackage{todonotes}

\newcommand{\bc}{\color{blue}}

\def\a{{\alpha}}
\def\b{{\beta}}
\def\g{{\gamma}}

\def\d{{\delta}}

\def\f{{\phi}}

\begin{document}

\newtheorem{theorem}{Theorem}[section]
\newtheorem{lemma}[theorem]{Lemma}
\newtheorem{proposition}[theorem]{Proposition}
\newtheorem{corollary}[theorem]{Corollary}
\newtheorem{problem}[theorem]{Problem}
\newtheorem{construction}[theorem]{Construction}

\theoremstyle{definition}
\newtheorem{defi}[theorem]{Definitions}
\newtheorem{definition}[theorem]{Definition}
\newtheorem{remark}[theorem]{Remark}
\newtheorem{example}[theorem]{Example}
\newtheorem{question}[theorem]{Question}
\newtheorem{comment}[theorem]{Comment}
\newtheorem{comments}[theorem]{Comments}

\newtheorem{discussion}[theorem]{Discussion}

\renewcommand{\thedefi}{}

\long\def\alert#1{\smallskip{\hskip\parindent\vrule%
\vbox{\advance\hsize-2\parindent\hrule\smallskip\parindent.4\parindent%
\narrower\noindent#1\smallskip\hrule}\vrule\hfill}\smallskip}

\def\ff{\frak}
\def\tf{torsion-free}
\def\fg{finitely generated }
\def\Spec{\mbox{\rm Spec}}
\def\Proj{\mbox{\rm Proj }}
\def\hgt{\mbox{\rm ht }}
\def\type{\mbox{ type}}
\def\Hom{\mbox{\rm Hom}}
\def\rank{\mbox{ rank}}
\def\Ext{\mbox{\rm Ext}}
\def\Tor{\mbox{\rm Tor}}
\def\Ker{\mbox{\rm Ker }}
\def\Max{\mbox{\rm Max}}
\def\End{\mbox{\rm End}}
\def\xpd{\mbox{\rm xpd}}
\def\Ass{\mbox{\rm Ass}}
\def\emdim{\mbox{\rm emdim}}
\def\epd{\mbox{\rm epd}}
\def\repd{\mbox{\rm rpd}}
\def\ord{\mbox{\rm ord}}
\def\Fdim{\mbox{\rm Fdim}}
\def\Mod{\mbox{\rm Mod-}}

\def\DD{{\mathcal D}}
\def\EE{{\mathcal E}}
\def\FF{{\mathcal F}}
\def\GG{{\mathcal G}}
\def\HH{{\mathcal H}}
\def\II{{\mathcal I}}
\def\LL{{\mathcal L}}
\def\MM{{\mathcal M}}
\def\PP{{\mathcal P}}

\title[De-noetherizing Cohen--Macaulay rings]{De-noetherizing Cohen--Macaulay rings}




\author{L\'aszl\'o Fuchs}
\address{Department of Mathematics, Tulane University, New Orleans, Louisiana 70118, USA}
\email{fuchs@tulane.edu}

\author{Bruce Olberding}
\address{Department of Mathematical Sciences, New Mexico State University, Las Cruces, New Mexico 88003, USA}
\email{olberdin@nmsu.edu}

\maketitle

\begin{center}
\today
\end{center}

 \begin{abstract}    We introduce a new class of  commutative {non-noetherian} rings, called $n$-subperfect rings,  generalizing the almost perfect rings that have been studied recently by Fuchs--Salce \cite{FS}.  For an integer $n \ge 0$, the ring $R$ is  $n$-subperfect if every  maximal regular sequence in $R$ has length $n$ and the total ring of quotients of $R/I$ for any ideal $I$ generated by a regular sequence is a perfect ring in the sense of Bass. We define an extended  Cohen--Macaulay ring as a  commutative ring $R$ that has noetherian prime spectrum and each localization $R_M$ at a maximal ideal $M$ is ht($M$)-subperfect. In the noetherian case,  these are precisely the classical  Cohen--Macaulay rings.  Several relevant properties are proved reminiscent of those shared by Cohen--Macaulay rings.
 
   \end{abstract} \medskip


 \section{Introduction} \medskip

   The Cohen--Macaulay rings play extremely important roles in most branches of commutative algebra. They have a very rich, fast expanding theory and a wide range of applications where the noetherian hypothesis is essential  in most aspects.     Cohen--Macaulay rings $R$ are usually defined in one of the following ways: \smallskip
    
      (a) $R$ is a noetherian ring in which ideals generated by elements of regular sequences are unmixed (i.e. have no embedded primes).
   
   (b) $R$ is a noetherian ring such that the grade (the common length of maximal regular sequences in $I$)  of every proper ideal $I$ equals the height of $I$.  \smallskip

    It is natural to {search for generalizations} to non-noetherian rings that still share many of the useful properties with  Cohen--Macaulay rings. As a matter of fact, there have been several attempts for generalization, a few reached publication, see \cite{G}, \cite{H}, \cite{HM}, \cite{AT}. In each of these generalizations, the noetherian condition was replaced by one without close connection to the noetherian property.   We believe that a generalization that is closer to the noetherian condition might allow for new applications and capture {somewhat} more features of Cohen--Macaulay rings than the generalizations in the cited references.

         In this note, 
     we are looking for a kind of generalization that is very natural and is as close to  Cohen--Macaulay rings as possible, but general enough to be susceptible of applications. We break tradition and choose a different approach: one  that does not adhere to any of the classical defining properties.  Our strategy is to rephrase the definition to one that does not  explicitly require the noetherian condition, to replace the condition that implies the noetherian character by a weaker one, and after doing so, to use the modified definition as the base of generalization.  
   
   It is not difficult to check that an ideal $I$ of a noetherian ring $R$   that is generated by a regular sequence  is unmixed if and only if the ring of quotients of $R/I$ is an artinian ring. It is likewise easy to see that for domains by dropping the noetherian requirement in (a) and replacing the unmixed condition by assuming the artinian property of the ring of quotients of $R/I$, we obtain a genuine characterization of noetherian  Cohen--Macaulay domains.
    Using this {observation} as a point of departure, we follow our strategy, and want to  de-noetherize the artinian property. But nothing is simpler than that: we just replace the descending chain condition on \emph{all} ideals by the descending condition on \emph{\fg} ideals. We do not stop here, but recall that the descending condition on \fg ideals is equivalent to the same condition on principal ideals \cite[Theorem 2]{Bj}, and the latter condition characterizes the \emph{perfect} rings, introduced by Bass \cite{Ba}. In conclusion, \emph{we will generalize  Cohen--Macaulay rings by replacing `artinian' by `perfect'.}  Accordingly, we will call a ring $R$ (with {maximal} regular sequences of lengths $n$) \emph{$n$-subperfect} {$(n \ge 0)$} if the ring of quotients of the ring $R/I$ is perfect for every proper ideal $I$ generated by a regular sequence, and add right away that a  $0$-subperfect ring is the same as a perfect ring in the sense of Bass. 
    
    By an \emph{extended  Cohen--Macaulay ring} we shall mean a commutative ring $R$ that has noetherian prime spectrum and each localization $R_M$ at a maximal ideal $M$ is ht($M$)-subperfect. In our discussion we will concentrate on the $n$-subperfect case {for a fixed $n \ge 0$} {(which is more general than the local case)}.   
   
     Asgharzadeh and Tousi \cite[Theorem 3.1]{AT} review and compare the various non-noetherian generalizations of Cohen--Macaulay rings in the literature and add their own variants. In a sense, our generalization lies properly between the classical Cohen--Macaulay rings and their generalizations in the literature, at least as far as zero-dimensional rings are concerned. In fact, a zero-dimensional ring is Cohen--Macaulay if and only if it is artinian, while each of the generalizations  listed  in \cite{AT} includes all zero-dimensional rings in their versions of generalized Cohen--Macaulay rings. In our generalization,  in the class of zero-dimensional rings only the perfect rings qualify.  ({A main} difference is in the nilradical: T-nilpotency is properly between being just nil and even nilpotent.) Furthermore, {the} one-dimensional integral domains are included in all of the previously published generalizations. For the  Cohen--Macaulayness however, such domains ought to have artinian factor rings modulo any non-zero ideal, while for our $1$-subperfectness these factors are required to be perfect rings.   Being closer to the classical version, our generalization is expected to share more analogous properties with Cohen--Macaulay rings than the previous generalizations, yet capture fewer classes of rings.  To avoid confusion involving these different generalizations of Cohen--Macaulay rings, we assume implicitly in what follows that  the term ``Cohen--Macaulay ring'' designates a {\it noetherian} Cohen--Macaulay ring.  
   
   Let us point out some relevant features of $n$-subperfect rings that support our claim that this generalization has a number of properties that are fundamental for Cohen--Macaulay rings in the noetherian setting. ($n$ can be any non-negative integer in the following list.)  
   \smallskip
   
       {$\bullet$ A ring $R$ is $n$-subperfect if and only if for each regular sequence $x_1,\ldots,x_i$ in $R$, the ring $R/(x_1,\ldots,x_i)R$ is $(n-i)$-subperfect  (Proposition \ref{factor prop}). }    
  
   $\bullet$  A ring $R$ is $n$-subperfect if and only if its spectrum is noetherian and the localizations $R_M$ are $n$-subperfect for all maximal ideals $M$ (Corollary~\ref{local global cor}). 
   
       $\bullet$ An $n$-subperfect ring is catenary, equidimensional, and of Krull dimension $n$ (Corollary \ref{Krull}).

  $\bullet$ A  noetherian  ring  is Cohen--Macaulay  if and only if 
  it is an extended Cohen--Macaulay ring (Corollary \ref{CM char}).

       $\bullet$    The grade of a proper ideal $I$ of an $n$-subperfect ring $R \ ($the length $t$ of the longest regular sequences contained in $I)$  is the smallest integer $t$ such that $\Ext_R^t(R/I, R) \neq 0$ (Theorem \ref{ideal}). 
        
         $\bullet$ If a finite group $G$ operates on an $n$-subperfect ring $R$ and its order is a unit in $R$, then the set $R^G$ of ring elements fixed under $G$ is {an $n$}-subperfect ring (Corollary \ref{inv}).

        $\bullet$  The polynomial ring $R[X_1, \dots, X_n]$, or any of its Veronese subrings, is $n$-subperfect if and only if $R$ is a perfect ring (Theorems \ref{poly n-subperfect} {and ~\ref{Veronese}). 
        
       {$\bullet$  The nilradical $N$ of an  $n$-subperfect ring $R$ is  T-nilpotent, and $R/N$ is a Goldie ring (Lemma \ref{H}, Theorem \ref{Goldie}).}
     \smallskip

     Our definition leaves ample room for specializations: additional conditions might be added that are not strong enough to enforce the noetherian property, but lead to more pleasant  properties of the resulting generalization (e.g. restriction of the  finitistic dimension or the $h$-local property might be such a condition).  {Examples for $n$-subperfect rings that are not Cohen--Macaulay are abundant; see Section 8.}

   Our main goal was to get acquainted with the fundamental properties  of $n$-perfect rings that are analogous to well-known features of Cohen--Macaulay rings. Working in the non-noetherian situation and in the uncharted territory of subperfect rings meant a challenge in several proofs. We focus our attention to $n$-subperfectness (which suffices to explore the general case) in order to avoid  dealing with the complicated general situation corresponding to global Cohen--Macaulay rings that would make the main features less transparent. Occasionally, when it does not obscure  
the main ideas, we work under the global analogue of Cohen-Macaulay rings; these are the regularly subperfect rings defined in Section 2.    (See Corollary~\ref{CM char}.)}
         
  While perhaps less familiar in commutative algebra, perfect rings, the cornerstone of our approach, appear throughout the literature on modules and associative algebras. We review these rings briefly in the next section, but see, for example, Bass \cite{Ba} and Lam \cite{La} for more background.    
   As an application of our approach, we obtain  a well-developed Cohen--Macaulay theory of regular sequences in polynomial rings over perfect rings. Thus, while perfect rings help illuminate the workings of Cohen--Macaulay rings, Cohen--Macaulay rings in turn {might} help shed new light on the  class of perfect rings.   
      
     \smallskip
     

\section{Definitions and Notations}   
  
   All rings considered here are commutative.   We mean by a {\it perfect ring} a ring over which  flat modules are projective.  
    Most of the following characterizations of perfect commutative rings can be found in Bass {\cite[Theorem P]{Ba}} and  Lam \cite[Theorems 23.20, 23.24]{La}. Recall that a module $M$ is {\it semi-artinian} if every non-zero {epic image} of $M$ contains a simple submodule.

   \begin{lemma}  \label{perf}  The following are equivalent for  a commutative ring $R$: \smallskip

{\rm (a)} $R$ is a perfect ring;
 
{\rm (b)} $R$ satisfies the descending chain condition on principal ideals;

{\rm (c)} $R$ is a finite direct product of local rings with T-nilpotent maximal ideals;

{\rm (d)} $R$ is semilocal and the localization $R_P$ is perfect for every maximal ideal $P$;

{\rm (e)} $R$ is semilocal and semi-artinian;

{\rm (f)} the finitistic dimension $\Fdim(R) \ ($supremum of finite projective dimensions of $R$-modules$)$ is $0;$

{{\rm (g)} the $R$-modules admit projective covers.}
  \qed
\end{lemma}

   A ring $R$ is {\it subperfect} if its total quotient ring $Q(R)$ is perfect, i.e. it is an order in a perfect ring. This is {a} most essential concept in this paper.  All Cohen-Macaulay rings are subperfect.
  Subperfect rings can be characterized  as follows.

\begin{lemma} \label{H}  For a commutative ring $R$, these are equivalent: \smallskip

{\rm (i)} $R$ is  subperfect. \smallskip

{{\rm (ii)} $R$ has only finitely many minimal prime ideals, every zero-divisor in $R$ is contained in a minimal prime ideal, and the nilradical $N$ of $R$ is T-nilpotent $($i.e. for every sequence $y_1, \dots , y_n \dots$ in $N$ there is an index $m$ such that $y_1 \cdots y_m=0).$}
 \smallskip
 
{\rm (iii)} {\rm (Gupta \cite{Gu})} $R$ satisfies:

\quad {\rm (a)} the nilradical  $N$ of $R$ is T-nilpotent,

\quad {\rm (b)} $R/N$ is a $($reduced $)$ Goldie ring $($i.e. it has finite  uniform dimension and satisfies the ascending chain condition on annihilators of subsets$)$,   and 

\quad {\rm (c)} a regular coset of $N$ can be represented by a regular element of $R$. {$($Moreover, a regular coset of $N$ consists of   regular elements of $R.)$} \smallskip 

{\rm (iv)} {\rm (Fuchs--Salce \cite[Theorem 6.5]{FS})} The modules over the quotient ring $Q(R)$ are weak-injective as $R$-modules.\smallskip 

{\rm (v)}   {\rm (Fuchs--Salce \cite[Theorem 6.4]{FS})} If $M$ is an $R$-module of weak dimension $\le 1$, then $Q(R)\otimes_R M$ is a $Q(R)$-projective module.
\qed \end{lemma}

 Here an $R$-module $M$ is said to be {\it weak-injective} if $\Ext_R^1(A,M)=0$ for all $R$-modules  $A$ of weak-dimension $\le 1$ (Lee \cite{L}).

  An ideal $I$ of the commutative ring $R$ 
   is {\it subperfect} if $Q(R/I)$ is a perfect ring, i.e., $R/I$ is a subperfect ring. A regular sequence is {\it subperfect} if the ideal it generates is   subperfect.  
We use the conventions that regular sequences are proper and that the empty sequence is considered a regular sequence. Thus   the empty sequence in $R$ is subperfect if and only if  $R$ is subperfect.

  We say a ring $R$ is {\it regularly subperfect} if each regular sequence of $R$ is subperfect.  Thus a ring $R$ is regularly subperfect if and only if  for each regular sequence $x_1,\ldots,x_i$ in $R$ (including the empty regular sequence), the ring $R/(x_1,\ldots,x_i)R$ is subperfect. In particular, a necessary condition for  $R$ to be regularly subperfect is that  $R$ itself  is subperfect.
For an integer $n \ge 0$, the ring $R$ is  {\it $n$-subperfect} if $R$ is regularly subperfect and every  maximal regular sequence has length $n$.
   As a consequence,  $R$ is $0$-subperfect if and only if $R$ is perfect. This is because in a $0$-subperfect ring every non-unit is a zero-divisor, so $Q(R) = R$. 
  
    
  We summarize several important properties of $n$-subperfect rings.  

\begin{quote} {\it  For each $n \ge 0$, the class of $n$-subperfect rings  is closed under finite direct sums; also under direct summands subject to integrality conditions. Moreover, localizations at prime ideals and factor rings modulo ideals generated by regular sequences are $m$-subperfect for some $m \le n$.}
\end{quote}

The assertion regarding direct summands is formalized in Theorem~\ref{summand}, while localization of $n$-subperfect rings is the subject of Section 4. The remaining assertions are more straightforward (see e.g.~Corollary \ref{direct}).  It need not be the case that a factor ring of a $n$-subperfect ring modulo a nil ideal is $m$-subperfect for some $m \leq n$. This fails even in the noetherian case; see  \cite[Exercise 2, p.~97]{HRW}.

   The 1-subperfect rings have been studied recently under the name {\it `almost perfect rings;'} see Fuchs--Salce \cite{FS} and Fuchs \cite{F}. They were defined as \emph{subperfect rings}  such that each factor {ring} modulo a regular ideal (i.e., an ideal containing a non-zero-divisor) is a perfect ring. 
   
   \begin{lemma}  \label{almperf} Suppose $R$ is a  subperfect ring. The following are equivalent: \smallskip
   
 $(\a)$ $R$ is almost perfect;
 
 $(\b)$ every non-zero torsion $R$-module contains a simple submodule;  
 
$(\g)$ for every regular proper ideal $I$ of $R$, $R/I$ contains a simple module;

 $(\d)$ $R$ is $h$-local and   $Q(R)/R$ is semi-artinian.
 \qed

 \end{lemma}

    Moreover, almost perfect rings (i.e. 1-subperfect rings)  have a number of interesting characteristic properties that are new even for  Cohen--Macaulay rings of Krull dimension 1. To wit, we mention the following \cite{FS}, \cite{F}:  
   A subperfect ring $R$ is almost perfect if and only if either of the following conditions is  satisfied: \smallskip
    
    {(i)} All flat $R$-modules are strongly flat (strongly flat means that it is {a summand of a module that is an extension of a free $R$-module by a direct sum of copies of}  the ring of quotients $Q$ of $R$).  
    
    {{\rm (ii)} $R$-modules of weak dimension $\le 1$ are of projective dimension $\le 1$. } 
        
    {\rm (iii)}  Every $R$-module $M$ has a divisible envelope (i.e. a divisible module containing $M$  and being contained in every divisible module that  contains $M$).

    {\rm (iv)}  {If $R$ is reduced:} each $R$-module $M$ admits a projective dimension $1$ cover (i.e. a module of projective dimension $\le 1$ along with a map $\a$ to $M$ such that any map from a module of projective dimension $\le 1$ to $M$ factors through $\a$, {and no proper summand has this property)}.  
\smallskip

For several   results in Section 3, as well in later arguments,  we work with  regular sequences  that generate ideals that are not necessarily subperfect.    We recall first some standard terminology.  
  Let $R$ be a ring (commutative), and $R^\times$ the set of regular (non-zero-divisor) elements of $R$.  
   An element $r $ of $R \setminus I$ is {\it prime to an ideal $I$} of $R$ if whenever $s \in R$ with $rs \in I$, then $s \in I$.  The set $S$ of elements prime to $I$ is  a saturated multiplicatively closed set. The prime ideals of $R$ that contain $I$ and are maximal with respect to not meeting $S$ are the {\it maximal prime divisors} of $I$. The prime ideals of $R$  that are minimal with respect to containing $I$ are the {\it minimal prime divisors} of $I$. These ideals do not meet $S$.  It follows that the classical ring of quotients $Q(R/I)$ of $R/I$ is $R_S/I_S$, and the maximal ideals of $Q(R/I)$ are the extensions to $Q(R/I)$ of the maximal prime divisors of $I$.  Similarly, the minimal prime ideals of $Q(R/I)$ are the extensions of the minimal prime divisors of $I$.  

We say an ideal $I$ of the ring $R$ is {\it unmixed} if every maximal  prime divisor  of $I$ is {also} a minimal prime divisor of $I$; equivalently, $\dim Q(R/I) = 0$. Thus,  $I$ is unmixed if and only if  every element  of $R$ not  in a minimal prime divisor of $I$ is prime to $I$.  In the case where $R$ is noetherian, this agrees with the  definition of unmixed ideal given by Bruns and Herzog in \cite[p.~{59}]{BH}. If $R$ is noetherian,  $Q(R/I)$ is semilocal. However, since  non-noetherian rings are our main focus, {in our discussions} $Q(R/I)$ need not be semilocal without additional assumptions on $I$.

 We say that an ideal $I$ of $R$ is 
 {\it finitely unmixed} if $Q(R/I)$ is a semilocal zero-dimensional ring. A regular sequence of $R$ is {\it finitely unmixed} if the ideal it generates is finitely unmixed. Thus every subperfect regular sequence is finitely unmixed, and every finitely unmixed regular sequence is unmixed.  

For unexplained terminology we refer to {Matsumara \cite{Ma} and} Bruns--Herzog \cite{BH}.


\section{Basic Properties} 

\medskip

Although the focus for most of the article is on $n$-subperfect rings, in this section we prove several assertions in greater generality. 

For an integer $n\geq 0$, say that a ring $R$ is {\it $n$-unmixed} if every regular sequence of $R$ extends to a maximal regular sequence of length $n$ that is unmixed.  Let ${\mathcal C}$ be a class of zero-dimensional rings.
 We call a ring $R$ is {\it $n$-unmixed} in ${\mathcal C}$  if
  every regular sequence  extends to a maximal regular sequence of length $n$ and for every regular sequence $x_1,\ldots,x_i$ in $R$, we have   $Q(R/(x_1,\ldots,x_i)R) \in {\mathcal C}$. Thus a ring $R$ is $n$-subperfect if and only if $R$ is $n$-unmixed in the class ${\mathcal C}$ of perfect rings.  
  
 The property of being $n$-unmixed in a class  ${\mathcal{C}}$ of zero-dimensional rings can be inductively described, as in the next lemma.    

\begin{lemma} \label{factor}  Let ${\mathcal{C}}$ be a class of zero-dimensional rings, and $n \geq 0$. 
A ring $R$ is  $n$-unmixed in  ${\mathcal{C}}$  if and only if for each $0 \leq i <n$ and {for each} regular sequence  $x_1,\ldots,x_i$  in $R$, the ring $R/(x_1,\ldots,x_i)R$ is  $(n-i)$-unmixed in  ${\mathcal{C}}$. 
\end{lemma}

\begin{proof}   Suppose $R$ is $n$-unmixed in  ${\mathcal{C}}$, and let $0 \leq i < n$.  Since $R$ is $n$-unmixed, every regular sequence that begins with $x_1,\ldots,x_i$ extends to a maximal regular sequence of length $n$. It follows that every maximal regular sequence in $R/(x_1,\ldots,x_i)R$ has length $n-i$. Also, since every regular sequence in $R$ is unmixed in  ${\mathcal{C}}$, so is every regular sequence in $R/(x_1,\ldots,x_i)R$. Thus, $R/(x_1,\ldots,x_i)R$ is  $(n-i)$-unmixed in  ${\mathcal{C}}$. 

Conversely, suppose that for each $0 \leq i <n$ and {for each} regular sequence  $x_1,\ldots,x_i$  in $R$, the ring $R/(x_1,\ldots,x_i)R$ is $(n-i)$-unmixed in  ${\mathcal{C}}$. Let $x_1,\ldots,x_i$ be a regular sequence in $R$,  and let $j \leq i$.  Then the zero ideal in $R/(x_1,\ldots,x_j)R$ is by assumption unmixed in  ${\mathcal{C}}$, so $Q(R/(x_1,\ldots,x_j)R) \in {\mathcal{C}}$. Moreover, since $R/(x_1,\ldots,x_i)R$ is $(n-i)$-unmixed, every maximal regular sequence in this ring has length $n-i$.  Thus every extension of $x_1,\ldots,x_i$ to a maximal regular sequence in $R$ has length $n$. This proves $R$ is $n$-unmixed {in ${\mathcal{C}}$}. 
\end{proof}

\begin{proposition} \label{factor prop}  {Assume}  $n \geq 0$. The ring $R$ is $n$-subperfect if and only if, for each regular sequence $x_1,\ldots,x_i$ in $R$, the ring $R/(x_1,\ldots,x_i)R$ is $(n-i)$-subperfect.
\end{proposition}

\begin{proof} Apply Proposition~\ref{factor}  to   the class ${\mathcal{C}}$ of perfect rings.  
\end{proof}

We record the following corollary that also shows how $n$-perfectness can be defined by  induction on $n$.

\begin{corollary} \label{reduction}   A ring $R$ is $n$-subperfect $(n \geq 1)$ if and only if it is subperfect and for each regular element $x \in R$, the ring $R/xR$ is $(n-1)$-subperfect.
\end{corollary}

\begin{proof} This is an immediate consequence of Proposition \ref{factor prop}. \end{proof}

 The property of being $n$-unmixed also has strong consequences for the dimension theory of the ring. 

\begin{proposition} \label{dim} Suppose $n \geq 0$. If the ring  $R$   is  $n$-unmixed, then $\dim R = n$ and all maximal chains of prime ideals of $R$ have the same length $n$.   
\end{proposition}

\begin{proof} 
We first prove by induction on $n$  that $\dim R = n$. 
If $n = 0$, then the empty regular sequence is unmixed, and so  $\dim Q(R) = 0$. {In this case regular elements are units,  therefore} we have $R = Q(R)$.  Thus, for $n = 0$,  $\dim R = 0$ and the claim is clear.

Suppose that $n>0$ and for each $0 \leq i < n$, every  $i$-unmixed ring has dimension $i$.  We claim that $\dim R = n$. Since $R$ is $n$-unmixed with $n>0$, we have $\dim R >0$.  Suppose that $P_0 \subset P_1 \subset \cdots \subset P_{m}$  is a chain of distinct prime ideals of $R$ with $m>0$. 
Since $R$ is $n$-unmixed with $n>0$, we have $R \ne Q(R)$  and $\dim Q(R)  =0$. Hence every ideal of $R$ not contained in a minimal prime ideal is regular, so there is a regular $x \in P_1$. By Lemma~\ref{factor},   $R/xR$ is $(n-1)$-unmixed.  
 By the induction hypothesis,  $\dim R/xR =  n-1$. Since $P_1/xR \subset \cdots \subset  P_{m}/xR$ is a chain of distinct prime ideals  of $R/xR$ and $\dim R/xR  = n-1$, we conclude that $m \leq  n$.  Thus  
  no chain of distinct prime ideals of $R$ has length {exceeding $n$, that is, $\dim R \leq n$.}  To see that $n \leq \dim R$, use the fact that  $R$ has a regular sequence of length $n$  \cite[Theorem 132]{K}.  Therefore, $\dim R = n$.

 Next we show that all maximal chains of prime ideals have the same length. The proof is again by induction on $n$. If $n = 0$, then, as we have established,  $\dim R = 0$. In this case the proposition is clear. Let $n>0$, and  {suppose the claim holds for all $i < n$.}    Let $P_0 \subset P_1 \subset \cdots \subset  P_k$ and $Q_0 \subset Q_1 \subset \cdots \subset Q_m$ be  maximal chains of distinct prime ideals in $R$.  We claim $k = m$. Since  the zero ideal of $R$ is unmixed, every non-minimal prime ideal of $R$ is regular. Thus $P_1$ and $Q_1$ are regular ideals of $R$,  so there is an  $x \in R^\times$  in $P_1 \cap Q_1$.  By Lemma~\ref{factor}, $R/xR$ is an $(n-1)$-unmixed ring  with maximal chains of prime ideals $P_1/xR  \subset \cdots \subset  P_k/xR$ and $Q_1/xR \subset \cdots \subset Q_m/xR$.  
By the induction hypothesis on $R/xR$,  we have $k-1 = m-1$, thus $k = m$.   This means that all chains of maximal length in $R$ have the same length $k$. It follows that $\dim R = k$, thus $k = \dim R = n$.   
\end{proof}

\begin{corollary} \label{Krull}  For every $n \ge 0$, an $n$-subperfect ring is catenary,   equidimensional, and has  Krull dimension $n$. \qed \end{corollary}

For an ideal $I$ of a ring $R$, the {\it $I$-depth} of $R$ 
 is the smallest positive integer $t$ such that $\Ext_R^t(R/I,R) \ne 0$.  If $R$ is noetherian, then the $I$-depth of $R$ is the length of the longest regular sequence contained in $I$.   
 Thus a noetherian ring $R$ is Cohen--Macaulay if and only if for each proper ideal $I$ of $R$, the $I$-depth of $R$ is equal to the height of $I$.  
  We show in Theorem~\ref{ideal} that this result holds more generally for regularly subperfect rings.



\begin{theorem} \label{ideal}
Let $R$ be a regularly subperfect ring,  $I$ a proper ideal of $R$, and let $n \geq 0$.  The following are equivalent: 
 \begin{enumerate}

\item $I$ has height $n$. 

\item  Every maximal regular sequence in $I$ has length $n$.  

\item There exists a maximal regular sequence in $I$ of length $n$. 

\item $n = \min \{t:\Ext^t_R(R/I,R) \ne 0\}.$
\end{enumerate}
\end{theorem}
 
 \begin{proof} 
 We first prove the equivalence of (1), (2) and (3).  
 Since the length of a regular sequence in $I$ is at most the height of $I$,  
it  suffices to show that if $x_1,\ldots,x_t$ is a regular sequence in $I$ such that $t < $ ht$(I)$, then $x_1,\ldots,x_t$ extends to a regular sequence in $I$ of length $t+1$.   Assume $x_1,\ldots,x_t$ is such a regular sequence. Since $Q(R/(x_1,\ldots,x_t)R)$ is semi-local and zero-dimensional, there are only finitely many minimal prime ideals $P_1,\ldots,P_m$ of $(x_1,\ldots,x_t)R$, and each element of $R$ not prime to $(x_1,\ldots,x_t)R$ is in one of the $P_j$.  As   
$I/(x_1,\ldots,x_t)R$ has positive height, $I \not \subseteq P_j$ for any $j$, so  $I \not \subseteq P_1 \cup \cdots \cup P_m$ by prime avoidance.  Consequently,   there exists  $x_{t+1} \in I$ prime to $(x_1,\ldots,x_t)R$.  Thus 
 $x_1,\ldots,x_t,x_{t+1}$ is a regular sequence, and the equivalence of (1), (2) and (3) follows. 
 
 To see that (4) implies (3), suppose $\Ext^n_R(R/I,R) \ne 0$, and $x_1,\ldots,x_n$ is a regular sequence {in $I$}. Let $J = (x_1,\ldots,x_n)R$.  By \cite[p.~101]{K}, 
 $$\Ext^n_R(R/I,R) \cong \Hom_R(R/I,R/J).$$
 From this isomorphism we conclude that there is $y \in R \setminus J$ such that $yI \subseteq J$, and hence the image of  $I$ in $R/J$   consists of zero-divisors. Thus $x_1,\ldots,x_n$ is a maximal  regular sequence  in $I$.  
 
Finally, to see that (3) implies (4),  suppose $x_1,\ldots,x_n$ is a maximal regular sequence in $I$. (Since $I$ has finite height, such a regular sequence must exist.) Then the image  of $I$  in $R/J$ consists of zero-divisors.  
Since $x_1,\ldots,x_n$ is  subperfect, the ring $Q(R/J)$ is perfect. Let $P$ be a prime ideal of $R$  containing $I$  such that $P/J$ extends to a maximal ideal of  $Q(R/J)$.  
A routine application of Lemma~\ref{perf}(c) and (e) {guarantees the existence of} 
%
%
 $x \in R \setminus J$ such that $xP \subseteq J$. Thus $xI \subseteq J$. Define a homomorphism $f:R/I \rightarrow R/J$ by $f(r+I) = rx+J$ for all $r \in R$. Then $f \ne 0$, and so by the above isomorphism  $\Ext^n_R(R/I,R) \ne 0$. If  {$t \leq n$ satisfies} $\Ext^t_R(R/I/R) \ne 0$, then since (3) implies (4), we have $x_1,\ldots,x_t$ is a maximal regular sequence in $I$. By the equivalence of (2) and (3), this yields $t = n$.  
 \end{proof} 
 
 \begin{remark}
{From the proof of Theorem~\ref{ideal} it is evident} that statements (1), (2) and (3)  remain equivalent if rather than assuming $R$ is regularly subperfect we assume only that every regular sequence is finitely unmixed.  
\end{remark}

 \begin{corollary} \label{Krull2} Let $n \geq 0$.  A ring $R$ is $n$-subperfect if and only if  $R$ is regularly subperfect and each maximal ideal of $R$ has height $n$. 
\end{corollary}

\begin{proof} If $R$ is $n$-subperfect, then each maximal ideal of $R$ has height $n$ by Corollary~\ref{Krull}. Conversely, if  $R$ is regularly subperfect and each maximal ideal has height $n$, then  
 every maximal regular sequence in $R$ has length $n$ by Theorem~\ref{ideal}. 
 \end{proof}

To verify that 
    a local  noetherian ring $R$ of dimension $d$   is Cohen--Macaulay, it is enough to exhibit just one regular sequence of length $d$.  
     By contrast, the following example shows that in a local domain $R$ of dimension $d$, the existence of a subperfect regular sequence of length $d$ is not sufficient to guarantee that the domain is $d$-subperfect.  
    
    \begin{example} \label{Kabele}
    In \cite[Example 5]{Kab}, Kabele constructs a local domain $R$ having the ring $S=k[[x,y,z]]$ as an integral extension, where $k$ is a field of characteristic $2$ with $[k:k^2] = \infty$ and $x,y,z$ are indeterminates for $k$.  The ring $R$ has the property that  $x,y$ is  not a regular sequence in $R$, but  $zR, (z,x)R$ and $(z,x,y)R$ are {distinct} prime ideals of $R$, and hence $z,x,y$ is a subperfect regular sequence in $R$. 
    Moreover, $\dim R = 3$ as  $S$ has dimension 3 and is integral over $R$. 
    Since $x,y$ is not a regular sequence and $x$ is a non-zero-divisor in $R$,  the image of $y$ in $R/xR$ is a zero-divisor. If $R$ is $3$-subperfect, then $R/xR$ is subperfect, so $y$ is in a minimal prime ideal $P$ of $xR$.  
   In this case, Corollary~\ref{Krull} implies that $\dim R/P = 2$.   Let $P'$ be a prime ideal of $S$ lying over $P$.   $S$ is integral over $R$, so $\dim S/P' =\dim R/P =2$ \cite[Theorem~47, p.~31]{K}.  Since $S$ is a catenary domain, this implies ht$(P') = 1$. However, $(x,y)S$ is a height $2$ prime ideal of $S$ contained in $P'$, a contradiction. 
    Therefore, $R$ is not $3$-subperfect  despite the fact that $R$ has a length 3 maximal regular sequence that is subperfect.  
    \end{example}
    

\section{Localization and globalization} 

In this section we consider localization and globalization of the $n$-subperfect property. 
In general, issues of localization involving regular sequences are complicated by the fact that a regular sequence in a localization at a prime ideal need not be the image of a regular sequence in $R$.  However, as we observe in the next lemma, this problem  can be circumvented for regularly subperfect rings. 

\begin{lemma}  \label{replace}
Let $R$ be a regularly subperfect ring, $P$ a prime ideal of $R$, and let $x_1,\ldots,x_n$ be a regular sequence in $R_P$. Then there is a regular sequence $y_1,\ldots,y_n \in P$ such that  $(x_1,\ldots,x_i)R_P = (y_1,\ldots,y_i)R_P$ for each $i =1,\ldots,n$.  
\end{lemma}

\begin{proof}  Let $I$ and $J$ be the ideals of $R$ defined by 
      \begin{center} $I = \{r \in R:(\exists s \in R \setminus P) \: rs \in x_1R\}$ \
       and \  $J = \{r \in R:(\exists s \in R \setminus P) \: rs \in x_1P\}$.  \end{center} 
       Then $IR_P = x_1R_P$ and $JR_P = IPR_P$. Moreover, $J \subset I$ {is a proper inclusion}, since the image of $x_1$ in $R_P$  is a non-zero-divisor. $Q(R)$ is zero-dimensional and semilocal, {so $R$ has finitely many minimal prime ideals $P_1,\ldots,P_m$} such that the set of  zero-divisors in  $R$ is $P_1 \cup \cdots \cup P_m$.    Since the image of $x_1$ in $R_P$  is a non-zero-divisor,  $I \not \subseteq P_j$ for any $j$. By prime avoidance, there is $y_1 \in I$ such that $y_1 \not \in J  \cup   P_1 \cup \cdots \cup P_m$.  Since $IR_P$ is a principal ideal and  the image of $y_1$ in $R_P$ is not in  $JR_P$, Nakayama's Lemma implies $x_1R_P =  IR_P = y_1R_P$.  By the choice of $y_1$, {we have $y_1  \in R^\times$}.   
       
       Now suppose $1 < t \leq n$ and there is a regular sequence $y_1,\ldots,y_{t-1}$ such that  $(x_1,\ldots,x_i)R_P = (y_1,\ldots,y_i)R_P$ for each $1 \leq i \leq t-1$.  
 Then $Q(R/(y_1,\ldots,y_{t-1})R)$ is semilocal and zero-dimensional, so repeating the argument from the first paragraph  for the ring $R/(y_1,\ldots,y_{t-1})R$ yields $y_t \in P$ such that $y_1,\ldots,y_{t-1}, y_t $ is a regular sequence  in $P$ and  $(y_1,\ldots,y_{t-1},y_t)R_P = (x_1,\ldots,x_{t-1},x_t)R_P$.  
   By induction, the proof is complete.       
\end{proof}

    \begin{theorem} \label{localization} Let $R$ be a regularly subperfect  ring. For each prime ideal $P$  of $R$, the ring $R_P$ is regularly subperfect.   
    \end{theorem} 
    
    \begin{proof}  Let $P$ be a prime ideal of $R$.  Since $Q(R)$ is zero-dimensional, $Q(R_P) = Q(R)_{R \setminus P}$; see \cite[Proposition~1 and Corollary~1]{Lip}. Thus $Q(R_P)$ is perfect  since $Q(R)$ is, and so $R_P$ is subperfect.   It follows that the localization of a regularly subperfect ring at a prime ideal has the property that the empty regular sequence is subperfect.

     We now prove the theorem by induction on the length of regular sequences in $R_P$.  
   Let $n>0$, and suppose that for every regularly subperfect  ring $S$ 
   and prime ideal $L$ of $S$, every regular sequence of length $< n$ in $S_L$ is subperfect. 
  Let $x_1,\ldots,x_n $ be  a sequence in $ R$ whose image  in $R_P$ is a regular sequence. 
      By Lemma~\ref{replace} there is $y \in R^\times$ such that $x_1R_P = yR_P$.  
         Since  $R/yR$ is regularly subperfect and the image of the sequence $x_2,\ldots,x_n$ in $R_P/x_1R_P = R_P/yR_P$ is a regular sequence of length~$n-1$, the induction hypothesis implies that  $R_P/x_1R_P$ is regularly subperfect. Therefore, 
        the image of the sequence $x_2,\ldots,x_n$ in $R_P$ is a subperfect regular sequence,  and hence so is the image of the sequence $x_1, x_2, \ldots,x_n$ in $R_P$. 
     \end{proof} 
        
        \begin{corollary} \label{localization cor} Let $R$ be a regularly subperfect ring. If $P$ is a prime ideal of finite height $n$, then $R_P$ is $n$-subperfect. 
        \end{corollary}
        
        \begin{proof} This follows from Theorems~\ref{ideal}  {and~\ref{localization}.}  
        \end{proof}

    \begin{corollary} \label{CM char}  The following are equivalent for a  noetherian ring $R$.
 
 \begin{enumerate} 
 
 \item $R$  is Cohen--Macaulay. 
 
 \item  $R$ is regularly subperfect.  
 
 \item $R_M$ is  {\rm ht}$(M)$-subperfect  for each maximal ideal $M$ of $R$. 

 \end{enumerate}
 \end{corollary} 
 
 \begin{proof}
  To see that (1) implies (2),   let $x_1,\ldots,x_n$ be a regular sequence in $R$, and let $0 \leq i\leq n$.  By the Unmixedness Theorem \cite[Theorem~2.1.6, p.~59]{BH}, $x_1,\ldots,x_i$ is unmixed (as is the empty regular sequence). Since $R$ is noetherian, the zero-dimensional ring $Q(R/(x_1,\ldots,x_i)R)$ is semilocal, hence artinian, hence perfect. Consequently, {the sequence} 
 $x_1,\ldots,x_n$ is subperfect.  
 
  That (2) implies (3) follows from   Corollary~\ref{localization cor}. For (3) implies (1),  observe that  {if $R_M$ is ht$(M)$-subperfect for a maximal ideal $M$,} then by {Theorem}~\ref{ideal} the maximal ideal of $R_M$ contains a maximal regular sequence of length equal to the height of $M$. Therefore, $R$ is Cohen--Macaulay. 
   \end{proof}

A topological space is {\it noetherian} if its open sets satisfy the ascending chain condition.  It follows that every closed subset of a noetherian space is a union of finitely many irreducible components. Thus, if $R$ is a ring for which $\Spec(R)$ is noetherian, then each proper ideal of $R$ has {but} finitely many minimal prime divisors.

    \begin{theorem} \label{local global} Let $R$ be a ring of finite Krull dimension. Then   
     $R$ is regularly subperfect if and only if $\Spec(R)$ is noetherian and    $R_M$ is regularly subperfect for each maximal ideal $M$ of $R$.    \end{theorem}

      \begin{proof}   
        Suppose  $R$ is regularly subperfect. 
    By   Theorem~\ref{localization}, $R_M$ is  regularly subperfect for each maximal ideal $M$ of $R$. 
       The proof that  $\Spec(R)$ is noetherian is by induction on $\dim R$.
     If $\dim R =0$, then $R$ is subperfect, hence perfect, since the ideal $(0)$ of $R$ is generated by the empty regular sequence; thus $\Spec(R)$ is noetherian {in this case}. Suppose $\dim R>0$, and  for each $0 \leq k < \dim R$  every $k$-dimensional regularly subperfect ring  has a noetherian spectrum. Since $R$ is subperfect,  
   $R$ has {only} finitely many minimal prime ideals $P_1,\ldots,P_m$. 
Thus $\Spec(R$) is a finite union of the closed sets consisting of the prime ideals containing a given minimal prime ideal $P_j$. To prove that $\Spec(R$) is noetherian, we need only verify that each of the spaces $\Spec(R/P_j$) is noetherian.  {A space is noetherian if and only if it satisfies the descending chain condition on closed sets,  therefore} we need only prove that every proper closed subset of $\Spec(R/P_j$) is noetherian. Every  proper closed subset of $\Spec(R/P_j$) is homeomorphic to a subspace of $\Spec(R/(rR+P_j)$) for some $r \in R \setminus P_j$. Therefore, we treat only  spectra of rings of the latter form. 

{Suppose $r \in R \setminus P_j$  for some  $1 \leq j \leq m$,} and choose $p_j \in R$ such that $p_j$ is contained in exactly the minimal prime ideals of $R$ that do not contain $r$.  (This is possible by prime avoidance and the fact that  there are only finitely many minimal prime ideals of $R$.)  In particular, $p_j \in P_j$.  
  {Evidently,  $r+p_j \not \in P_1 \cup \cdots \cup P_m$, so that} $r+p_j \in R^\times$.  Thus $R/(r+p_j)R$ inherits from $R$ the property that each  regular sequence is subperfect.  
    By the induction hypothesis, $\Spec(R/(r+p_j)R)$ is a noetherian space. {As a subspace of a noetherian space,}  $\Spec(R/(rR+P_j))$ is noetherian.  {This completes the proof} that $\Spec(R$) is a noetherian space.

    Conversely, suppose   $\Spec(R)$ is noetherian, and $R_M$ is  regularly subperfect for each maximal ideal $M$ of $R$. Let $x_1,\ldots,x_t$ be a (possibly empty) regular sequence in $R$, and let $I =(x_1,\ldots,x_t)R$.   For each maximal ideal $M$ containing $I$,  the images of $x_1,\ldots, x_t$ in $R_M$ form a regular sequence, so $R_M/IR_M$ is subperfect by assumption.    
    We claim that $Q(R/I)$ is zero-dimensional. Let $r,s \in R$ such that $rs \in I$ and  $r$ is not contained in any minimal prime ideal of $I$. It suffices to show that $s \in I$.  If $M$ is any maximal ideal of $R$ containing $I$, then since $R_M/IR_M$ is subperfect and $rR_M$ is not a subset of  any minimal prime ideal of $IR_M$, we have $sR_M \subseteq   IR_M$.  Since this is true for each maximal ideal $M$ containing $I$, we conclude that $s \in I$. This proves that every zero-divisor in $R/I$ is contained in a minimal prime ideal of $R/I$. Therefore, $Q(R/I)$ is zero-dimensional.  
    
    Since $\Spec(R)$ is noetherian, $I$ has only finitely many minimal prime ideals $P_1,\ldots,P_m$, so $Q(R/I)$ is also semilocal. For each $j$, $R_{P_j}/IR_{P_j}$ is T-nilpotent, so it follows that $Q(R/I)$ has T-nilpotent nilradical, and hence $Q(R/I)$ is perfect.  This proves that every regular sequence in $R$ (including the empty sequence) is subperfect. 
    \end{proof}

    \begin{corollary} \label{local global cor} Assume $n \geq 0$. A ring $R$ is $n$-subperfect if and only if $\Spec(R)$ is noetherian and $R_M$ is $n$-subperfect for each maximal ideal $M$ of $R$.
    \end{corollary}
    
    \begin{proof} 
    Apply Corollary~\ref{Krull2} and Theorem~\ref{local global}. 
    \end{proof}

\begin{remark} \label{generality} The proofs  of Lemma~\ref{replace} and Theorems~\ref{localization} and~\ref{local global}  show that in the hypotheses of these results the property of being regularly subperfect can be replaced by the more general condition that every regular sequence is finitely unmixed. 
\end{remark}

We record an immediate consequence of Corollary \ref{local global cor}:

    \begin{corollary} \label{direct}  The direct product of a finite number of $n$-subperfect   rings is likewise  $n$-subperfect. \qed  \end{corollary}


\medskip
\section{More on $n$-Subperfect Rings} 

\medskip

We would like to point out  several important properties that are shared by $n$-subperfect rings with Cohen--Macaulay rings.  The first of these properties, proved by Hochster--Eagan \cite{HE} for Cohen--Macaulay rings, concern descent of the $n$-subperfect property to direct summands   and {to}  rings of  invariants of $n$-subperfect rings.

\begin{theorem}\label{summand} Let $R$ be a $n$-subperfect ring for some $n \geq 0$. If  $S$ is a subring of  $R$ such that $R$ is integral over $S$ and $S$ is a direct summand of $R$ as an $S$-module, then $S$ is $n$-subperfect.  
\end{theorem}

 \begin{proof}  
First we claim that $S$ is subperfect.  
Every minimal prime ideal of $S$ is contracted from a minimal prime ideal of $R$. Since $R$ is subperfect, there are {but}  finitely many minimal prime ideals of $R$, so there are {only} finitely many minimal prime ideals of $S$. 
Moreover, every zero-divisor in $R$ is an element of a minimal prime ideal of  the subperfect ring $R$, so the same holds for $S$. 
 Since the nilradical of $S$ is contained in that of $R$, it is T-nilpotent.  Consequently, $S$ is subperfect.

The proof proceeds now  by induction on $n$. Suppose $n  =0$, so that $R$ is perfect.
Then $\dim R  = 0$, and since $R$ is integral over $S$, we have $\dim S = 0$. Since $S$ is subperfect, this implies $S$ is perfect, {i.e.}  $0$-subperfect.

 {Now suppose $n>0$ and that  the claim holds for $n-1$}.  If every non-zero-divisor of $S$ {were} a unit, then since $S$ is subperfect, {we would have $\dim S = 0$.} $R$ is integral over $S$, whence $\dim R = 0$ {would  follow}. However, $R$ is $n$-subperfect, so $\dim R  = n>0$ by Corollary~\ref{Krull}. Therefore, there exist regular sequences in $S$ of length $>0$.  Let $s \in S^\times$ be a non-unit in $S$.  
   Since $S$ is a summand of $R$, it follows that $sR \cap S = sS$; see 
    \cite[Lemma~6.4.4]{BH}.  Thus $S/sS$ can be viewed as a direct summand of $R/sR$. Moreover, $R/sR$ is integral over $S/sS$.  
    
       To see that $s \in R^\times$, suppose to the contrary that $s$ is a zero-divisor in $R$. Since $R$ is subperfect, $s$ {is contained} in a minimal prime ideal $P_0$  of $R$.  By Corollary~\ref{Krull2}, there is a chain of distinct prime ideals $P_0 \subset P_1 \subset \cdots \subset P_n$, with $P_n$ a maximal ideal of $R$.  
     Since $R$ is integral over $S$, the chain $P_0 \cap S \subset P_1 \cap S \subset \cdots \subset P_n \cap S$ has length $n$.  Again since $R$ is integral over $S$, each chain of prime ideals of $S$ has a chain of prime ideals in $R$ lying over it. Therefore,  Corollary~\ref{Krull2} implies that the length of the longest chain of prime ideals in $S$ is $n$. Consequently, $P_0 \cap S$ is a minimal prime ideal of $R$. However, $s \in P_0 \cap S$ and $s \in S^\times$, a contradiction that implies $s \in R^\times$.   
  
  In view of  $s \in R^\times$, we have  
     $R/sR$ is  $(n-1)$-subperfect by Proposition~\ref{factor prop}.  By the induction hypothesis,  $S/sS$ is $(n-1)$-subperfect.  
 Since this is the case for all non-units $s \in S^\times$, Corollary~\ref{reduction} implies $S$ is $n$-subperfect, {completing}  the induction. 
  \end{proof} 

\begin{corollary}\label{inv} Assume $G$ is a finite group acting on an $n$-subperfect  ring $R$, and the order of $G$ is a unit in $R$. Then the set of invariants,
$$R^G = \{r \in R  : \ g(r) = r \ for \ all \ g \in G\},$$
is again an $n$-subperfect ring.
\end{corollary}

 \begin{proof} 
As in Bruns--Herzog \cite[pp.~281--283]{BH}, the hypotheses imply that $R^G$ is a direct summand of $R$ and $R$ is integral over $R^G$.  Thus we may apply Theorem~\ref{summand} to obtain the corollary. 
\end{proof} 
 
Lemma \ref{H} makes it possible to get more information on $n$-subperfect rings once we know more about Goldie rings.
   
  A commutative reduced  Goldie ring $R$ is an order in a semisimple ring $Q$ that is the direct product of fields $Q_j$, 
  $$Q = Q_1 \times \dots \times Q_m $$
 (see Lam \cite[Proposition 11.22]{La1}). If $X_j= \sum_{i \ne j} Q_i$, then  $P_j = X_j \cap R \ (j=1, \dots, m)$ is the set of minimal primes of $R$. Furthermore, each $R/P_j$ is an integral domain with $Q_j$ as quotient field. Recall that orders $R,R'$ in a ring $Q$ are {\it equivalent} if $qR {\subseteq} R'$ and $q'R' {\subseteq} R$ for some units $q,q' \in Q$. 

 {\begin{theorem}   \label{Goldie}   A reduced  $ n$-subperfect ring $R$  is a Goldie ring.  {It is  a  subdirect product of a finite number of integral domains of Krull dimension $n$. This subdirect product is equivalent to the direct product of the components.}

   \end{theorem}
  
 \begin{proof}   Assume $R$ is reduced and $n$-subperfect; in view of Lemma \ref{H}, it is a Goldie ring. It has but a finite number of minimal prime ideals $P_1, \dots, P_m$. From  $\cap_{j} P_j =0$ it follows that $R$ is a subdirect product of the integral domains $D_j=R/P_j$ (with quotient fields $Q_j$).   It is clear that $\dim D_j =n$ for each $j$. 
 
 Suppose $x_j \in P_i$  for all  $i \ne j$, but $x_j \notin P_j$. Then $x = \sum_j x_j \in R$ is a regular element, as it is not contained in any $P_j$. Therefore, $x= (x_1+P_1, \dots, x_m+P_m) \in D_1 \oplus \dots \oplus D_m$ is a unit in $Q$. Hence we conclude that $R$ and $R'= D_1 \oplus \dots \oplus D_m$ are equivalent orders in $Q$. 
  \end{proof}}
  
     We observe that {Theorem \ref{Goldie} holds also for the factor ring $R/N$ of  an $n$-subperfect ring $R$  modulo its nilradical $N$, though $R/N$ need not be $n$-subperfect. Note that this factor ring} is restricted in size inasmuch as $R/N$ must have finite uniform dimension. {On the other hand,} Example \ref{matrix} will show that the nilradicals {of $n$-subperfect rings} can have arbitrarily large cardinalities.
     
     We have failed to establish a stronger result in the preceding theorem (viz. that the domains $D_j$ are also $n$-subperfect), because passing modulo a minimal prime ideal, regular sequences do not map in general upon regular sequences, though the converse is true for all regularly subperfect rings as is shown by: 
     
 \begin{lemma} \label{mod P}
Let $R$ be a regularly subperfect ring, and let $P$ be a minimal prime ideal of $R$.  Then for  every regular sequence $y_1,\ldots,y_t$ in $S=R/P$, there is a regular sequence  $x_1,\ldots,x_t \in R$  such that $(x_1,\ldots,x_t)S = (y_1,\ldots,y_t)S$.  
\end{lemma}

 \begin{proof}  The proof is by induction on the length of the regular sequence. The claim is clearly true for the empty regular sequence. Suppose that $t\geq 0$ and the claim is true for all regular sequences in $S$ of length $t$. Let $y_1,\ldots,y_t,y_{t+1}$ be a regular sequence in $S$.  Then there is a regular sequence 
  $x_1,\ldots,x_t$  in $ R$   such that $(x_1,\ldots,x_t)S = (y_1,\ldots,y_t)S$.  
    Since $R/(x_1,\ldots,x_{t})R$ is subperfect,  {$(x_1,\ldots,x_t)R$ has  but a finite number of minimal prime ideals $L_1,\ldots,L_k$}. Let $x_{t+1} \in R$ such that $x_{t+1} + P = y_{t+1}$.   
    We observe that $P+x_{t+1}R \not \subseteq L_i$ for any $i$. Indeed,
     if $P  \subseteq L_i$ for  some $i$, then $L_i$ is a minimal prime ideal of $(x_1,\ldots,x_t)R + P$. 
     In this case, since $y_1,\ldots,y_{t+1}$ is a regular sequence in $S$ and $(x_1,\ldots,x_t)S = (y_1,\ldots,y_t)S$, it {is impossible to have}   $y_{t+1} \in L_i/P$.  Thus 
      $x_{t+1} \not \in L_i$ {which} shows that $P+x_{t+1}R \not \subseteq L_i$ for every~$i$. By a version of prime avoidance \cite[Theorem 124]{K}, this implies there is  $p \in P$ such that  $x_{t+1} - p \not \in L_i$ for each $i$.  Since $L_1,\ldots,L_k$ are the minimal prime ideals of $(x_1,\ldots,x_t)R$ and $R/ (x_1,\ldots,x_{t})R$ is subperfect, it follows that $x_1,\ldots,x_t,x_{t+1} - p$ is  a regular sequence in $R$ {such that}  $(x_1,\ldots,x_t,x_{t+1}-p)S = (y_1,\ldots,y_{t+1})S$.  This completes the induction and shows that every ideal of $S$ generated by a regular sequence is the image of an ideal of $R$ that is generated by a regular sequence.  
 \end{proof}

     \smallskip

The next theorem shows that for regularly subperfect rings, ideals of the principal class (i.e., ideals $I$ generated by  ht$(I)$ elements) behave like ideals in Cohen--Macaulay rings. 
As we point out after the theorem, this allows us to connect our version of Cohen--Macaulay rings to a  versatile and quite general notion of non-noetherian Cohen--Macaulay rings due to Hamilton and Marley.  
  
\begin{theorem} \label{height} Let $R$ be a regularly subperfect  ring, and let $I$ be an  ideal of $R$ generated by $t$ elements. The following are equivalent:
\begin{itemize}
\item[(1)]  $I$ has height~$t$. 
\item [(2)] $I$ has height at least $t$. 
\item[(3)]  $I$ is generated by a regular sequence of length $t$. 
\end{itemize}
\end{theorem}

\begin{proof}  That (1) implies (2) is clear, and that (3) implies (1) 
    follows from Theorem~\ref{ideal}. To see that (2) implies (3),  suppose ht$(I) \geq t$.  If  ht$(I) = 0$, then $I$ is generated by the empty regular sequence. The proof now proceeds by induction on ht$(I)$. Suppose that {in a  regularly subperfect  ring, every ideal $I=(x_1,\ldots,x_t)R$} of height at least  ht$(I) -1$ generated by ht$(I) -1$ elements is generated by a regular sequence of length ht$(I) -1$.   
{As a subperfect ring, $R$  admits only finitely many minimal prime ideals $P_1,\ldots,P_m$}.  Prime avoidance and the fact that   ht$(I)  >0$ imply that $I \not \subseteq P_1 \cup \cdots \cup P_m$.  By \cite[Theorem 124, p.~90]{K}, there exist $r_2,\ldots,r_t \in R$ such that $x:=x_1 + r_2x_2 + \cdots  + r_tx_t \not \in P_1 \cup \cdots \cup P_m$. Since $R$ is subperfect, $x \in R^\times$.  Moreover, $I = (x,x_2,\ldots,x_t)R$. 
  In order  to apply the induction hypothesis, we  consider next the ring $R/xR$.  
  
  Let 
 $P$ be a minimal prime ideal of $I$ such that ht$(P) = $  ht$(I) $.  
 By Theorem~\ref{localization}, $R_P$ is ht$(I) $-subperfect, so Proposition~\ref{factor prop} implies $R_P/xR_P$ is $($ht$(I) -1)$-subperfect.  By Corollary~\ref{Krull},  $\dim R_P/xR_P = $ ht$(I) -1$, and so $P/xR$ 
has height ht$(I) -1$ in $R/xR$.
Consequently, $P/xR$ is a minimal prime ideal of $I/xR$ of  height ht$(I) -1$ in $R/xR$.
 Thus $I/xR$ is an ideal of $R/xR$ that is generated by $t-1$ elements and has height at least ht$(I)-1$.
  By the induction hypothesis, $I/xR$ is generated by a regular sequence in $R$ of length $t-1$.  Thus $I$ is generated by a regular sequence of length $t$.  This proves that every  ideal of $R$ of height at least $t$ generated by $t$ elements is generated by a regular sequence of length $t$. 
  Consequently, (2) implies~(3).  
\end{proof}

Hamilton and Marley \cite[Definition 4.1]{HM} define a ring $R$ to be Cohen--Macaulay if every ``strong parameter sequence'' on $R$ is a regular sequence. The notion  of a 
 strong parameter sequence, which is defined via homology and cohomology of  appropriate Koszul complexes,  is beyond the scope of our paper. However, we may use Theorem~\ref{height} to show that the  regularly subperfect   rings  are Cohen--Macaulay in this sense.  To prove this, by \cite[Proposition 4.10]{HM} it suffices  to show that every height~$t$ ideal generated by $t$ elements is finitely unmixed.

\begin{corollary} \label{Hamilton} Every regularly subperfect ring  is Cohen--Macaulay in the sense of Hamilton and Marley. 
\end{corollary}

\begin{proof} 
Let $I$ be a height $t$ ideal of $R$ that is  generated by $t$ elements. As discussed before the corollary,  it suffices to observe that $R/I$ is subperfect, and this is the case since $R$ is regularly subperfect and {by Theorem~\ref{height}}   $I$ is generated by a regular sequence. 
\end{proof}
  

 \section{Polynomial Rings}  \medskip  
 
 We consider next  polynomial rings $S = R[X_1,\ldots,X_n]$ over a perfect ring $R$. Theorem~\ref{poly n-subperfect}, which shows such rings are $n$-subperfect, depends on the following lemma.

 

\begin{lemma} \label{second nil} \label{fu perfect} Let  $S$ be a finitely generated algebra over a perfect ring $R$.  For each proper ideal $I$  of $S$,  the nilradical of $S/I$ is T-nilpotent. If also $\dim Q(S/I) = 0$, then $S/I$ is subperfect. 
\end{lemma}

\begin{proof}  Let $I$ be a proper ideal of $S$. Then the nilradical of $S/I$ is $\sqrt{I}/I$, so to show that this nilradical is T-nilpotent, it suffices to show that for all $a_1,a_2,a_3,\ldots \in \sqrt{I}$, there exists $m>0$ such that $a_1a_2 \cdots a_m \in I$.   
We claim first that  there is $k>0$ such that $(\sqrt{I})^k \subseteq I + JS$, where $J$ denotes the Jacobson radical of  $R$. Since $R/J$ is an artinian ring (it is a product of finitely many fields) and $S/JS$ is a finitely generated $R/J$-algebra, the ring $S/JS$ is noetherian. Thus
the image of the ideal $\sqrt{I}$ in $S/JS$ is finitely generated. {Letting $f_1,\ldots,f_t \in \sqrt{I}$ such that $\sqrt{I} = (f_1,\ldots,f_t)S + JS$, and choosing $k>0$  such that $(f_1,\ldots,f_t)^kS \subseteq I$, we obtain $(\sqrt{I})^k \subseteq I+ JS$.}

  For each $i \geq 0$, we have  $a_{ik+1}a_{ik+2}  \cdots a_{ik+k} \in I + JS$, and so there is a finitely generated ideal {\bc $A_i \subseteq J$} such that $ a_{ik+1}a_{ik+2}  \cdots a_{ik+k} \in I +A_iS$.  
  {As a perfect ring, } $R$ satisfies the descending chain condition on finitely generated ideals \cite[Theorem~2]{Bj}, thus there is $t>0$ such that $A_1 A_2 \cdots A_t = A_1A_2 \cdots A_{t+1}$.  Since $A_{t+1} \subseteq J$, Nakayama's Lemma implies $A_1A_2 \cdots A_t = 0$.  
It follows that $$a_1 a_2 \cdots a_{tk+k} \in (I+A_0S)(I+A_1S) \cdots (I+A_tS) \subseteq  I,$$  which proves the first assertion. 
  
  Now suppose $\dim Q(S/I)=0$. Since  $R$ is perfect, $\Spec(R)$ is a finite, hence noetherian, space. As a finitely generated algebra over a ring with noetherian prime spectrum, $S$ also has noetherian prime spectrum \cite[Theorem 2.5]{OP}.   Hence $I$ has finitely many minimal prime divisors, and so, since $Q(S/I)$ is zero-dimensional, it follows that $Q(S/I)$ is semilocal. The nilradical of $Q(S/I)$ is  {T-nilpotent as it is extended from the T-nilpotent nilradical of $S/I$;   hence}  $Q(S/I)$ is perfect. 
\end{proof}
  


 We now prove the main theorem of this section. 
Statement (4) of Theorem~\ref{perfect order}, which is a byproduct of our arguments involving polynomial rings,  can be viewed as a characterization of a perfect ring in terms of its multiplicative lattice of ideals.

\begin{theorem} \label{perfect order} \label{poly n-subperfect} Let $R$ denote a semilocal zero-dimensional ring, and let $X_1,\ldots,X_n$ be indeterminates for $R$. Then the following are equivalent: 
\begin{itemize}
\item[(1)] $R$ is perfect.
\item[(2)] $R[X_1,\ldots,X_n]$ is subperfect. 
\item[(3)] 
$R[X_1,\ldots,X_n]$ is $n$-subperfect.
\item[(4)] For each sequence  $\{I_i\}_{i=1}^\infty$ of finitely generated subideals of the Jacobson radical $J$ of $R$ there exists $k>0$ such that $I_1 I_2 \cdots I_k = 0$. 

\end{itemize}
\end{theorem} 

 \begin{proof} 
Let $S = R[X_1,\ldots,X_n]$, and let $J$ denote the Jacobson radical ($=$ the nilradical) of $R$.    

(1) $\Rightarrow$ (4) Let  $\{I_i\}_{i=1}^\infty$ be a sequence of finitely generated subideals of $J$. Since $R$ is perfect, $R$ satisfies the descending chain condition on finitely generated ideals \cite[Theorem 2]{Bj}, thus there is $k>0$ such that $I_1 I_2 \cdots I_k = I_1I_2 \cdots I_{k+1}$.  Since $I_{k+1} \subseteq J$, Nakayama's Lemma implies $I_1I_2 \cdots I_k = 0$.  

(2) $\Rightarrow$ (4) 
Suppose there is a sequence $\{I_i\}_{i=1}^\infty$ of finitely generated subideals of $J$ such that for each $k>0$, $I_1 I_2 \cdots I_k \ne 0$. To show that $S$ is subperfect, we first  construct a sequence $\{f_i:i \in {\mathbb{N}}\}$ of polynomials in $J[X_1]$ such that for each $k>0$, $f_1 f_2 \cdots f_k \ne 0$.
{For each $i$, let $A_i = \{r_{i1},r_{i2},\ldots,r_{in_i}\}$ be a minimal set of (necessarily distinct) generators for $I_i$.} 
  Choose a sequence of positive integers, each a power of $2$,  such that $$e_{11} < \cdots < e_{1n_1} < \cdots < e_{i1} < \cdots < e_{in_i} < \cdots.$$  For each $i \in {\mathbb{N}}$, define $$f_i = r_{i1}X_1^{e_{i1}} + r_{i2} X_1^{e_{i2}} + \cdots + r_{in_i} X_1^{e_{in_i}}.$$  
  
  Suppose by way of contradiction that  there is  $k>0$ such that $f_1f_2 \cdots f_k =0$.  
 For each $i$ and $a \in A_i$, let $e(a)$ be the power of $2$ associated to $a$; i.e., if $a = r_{ij}$, then $e(a) = e_{ij}$. 
 (Since for each $i$,  the $r_{ij}$, $1 \leq j \leq n_i$, are distinct, this assignment of a power of $2$ to each $a_i$ is well defined.) 
{By the choice of the $e_{ij}$, the assignment $e:A_i \rightarrow {\mathbb{N}}$ is injective for each $i$.}

For each
 $k$, let $B_k =A_1 \times A_2 \times \cdots \times A_k$,   
Then $$f_1f_2 \cdots f_k  \:\: =  \sum_{(a_1,\ldots,a_k) \in  B_k} a_1 a_2 \cdots a_kX_1^{e(a_1) + e(a_2) + \cdots + e(a_k)}.$$ 
Suppose $(a_1,a_2, \ldots,a_k),(a'_1,a'_2,\ldots,a'_k) \in B_k$ such that  $$e(a_1) + e(a_2) + \cdots + e(a_k) = e(a'_1) + e(a'_2) + \cdots + e(a'_k).$$ By the choice of the $e_{ij}$, we have \begin{center} $e(a_1) < e(a_2) < \cdots < e(a_k) $ \:\: and \:\: $e(a'_1) < e(a'_2) < \cdots < e(a'_k).$ \end{center}
Since each natural number can be expressed uniquely as  a sum of {distinct} powers of $2$, it 
follows that for each $i \leq k$, we have $e(a_i) = e(a'_i)$.  Since $e$ is injective on $A_i$, we conclude that $a_i = a'_i$. 
This shows that each pair of  distinct sequences in $B_k$ yields distinct powers of $X_1$ in the product $f_1 f_2 \cdots f_k$.  Consequently, 
  $f_1 f_2 \cdots f_k = 0$ if and only if 
$a_1 a_2 \cdots a_k = 0$ for all $(a_1,a_2, \ldots,a_k) \in B_k$.   Similarly,  the product   $I_1I_2 \cdots I_k$ is $0$ if and only if $a_1a_2 \cdots a_k =0$ for all $(a_1,a_2,\ldots,a_k) \in B_k$.    
{But} we have assumed that $I_1 I_2 \cdots I_k \ne 0$ for each $k>0$, so we conclude that $f_1  f_2 \cdots f_k \ne 0$ for each $k>0$, as claimed. 

To see now  that $S$ is not subperfect, observe that $J[X_1]$ is a subset of  the nilradical of $Q(S)$.   {The elements $f_1,f_2,\ldots \in J[X_1]$ have the property that no product of finitely many of them is $0$, therefore the ring $Q(S)$ is not perfect.}

(4) $\Rightarrow$ (2) 
Let
 $f_1/g_1,f_2/g_2, \ldots $ be elements of the nilradical of $Q(S)$, where each $f_i \in S$ and each $g_i$ is a non-zero-divisor in $S$. 
 Then $f_1,f_2,\ldots$ are in the nilradical of $S$, which, since $S$ is a polynomial ring, is the extension  $JS$ of the nilradical {$J$} of $R$ to $S$.  
   {The ideal $I_i$  generated by the coefficients occurring in $f_i$ is contained in the nilradical of $R$, so 
  by} assumption, there is $k>0$ such that $I_1 I_2 \cdots I_k=0$.  {Since 
  $f_1f_2\cdots f_k \in I_1 I_2 \cdots I_kS$,  we have $f_1f_2 \cdots f_k = 0$,  thus} the nilradical of $Q(S)$ is T-nilpotent.  
 Furthermore, since $R$ is zero-dimensional, so is $Q(S)$ by \cite[Proposition 8]{Ara}. 
 Each prime ideal $L$ in $Q(S)$ contracts to one of the prime ideals $P$ in $R$. Since $PQ(S) \subseteq L$ is a prime ideal of $Q(S)$ and $Q(S)$ is zero-dimensional, it follows that $PQ(S) = L$. Therefore, since $R$ is semilocal, so is $Q(S)$.   This shows that $Q(S)$ is a zero-dimensional semilocal ring with T-nilpotent nilradical; i.e., $Q(S)$ is perfect.

(2) $\Rightarrow$ (3) 
 Suppose  $S $ is subperfect. Let $f_1,\ldots,f_t$ be a regular sequence in $S$, and let $I = (f_1,\ldots,f_t)S$. 
  It is shown in \cite{OUnmix} that since $R$ is zero-dimensional and semilocal and $I$ is generated by a regular sequence, the ring $Q(S/I)$ is zero-dimensional and semilocal. By Lemma~\ref{fu perfect}, $Q(R/I)$ is a perfect ring, {establishing} that $R$ is $n$-subperfect. 

(3) $\Rightarrow$ (2) 
This is clear. 

 (2) $\Rightarrow$ (1) 
   First observe that since $R$ is zero-dimensional,  if $P$ is a prime ideal of $R$, then $Q(S)$ has a prime ideal lying over $P$: take any prime ideal of $Q(S)$ that survives in the localization $Q(S)_{R \setminus P}$.     Thus every prime ideal of $R$ survives in $Q(S)$. As the composition of the flat extensions $R \subseteq S$ and $S \subseteq Q(S)$, the extension $R \subseteq Q(S)$   is thus  faithfully flat. In particular, $IQ(S) \cap R = I$ for all ideals $I$ of $R$.  If $I_1 \supseteq I_2 \supseteq \cdots $ is a chain of principal ideals in $R$, then {by the perfectness of  $Q(S)$} there exists $k >0$ such that $I_kS = I_{k+i}Q(S)$ for all $i>0$. Consequently, $I_k = I_kQ(S) \cap R = I_{k+i}Q(S) \cap R = I_{k+i}$ for all $i >0$, proving that $R$ is perfect:  the principal ideals satisfy the descending chain condition.
\end{proof}

Let us point out that Coleman--Enochs \cite{CE0} prove that the polynomial rings $R[X]$ and $R'[Y]$  with single indeterminates over perfect rings $R, R'$ are isomorphic if and only if $R \cong R'$. It is an open problem if this holds for more indeterminates.


  \section{The Finitistic Dimension} \medskip 
  
   The close relation of $n$-subperfect rings to Goldie rings makes it possible to derive several interesting properties of $n$-subperfect rings. For details we refer to the literature on Goldie rings, e.g. Goodearl--Warfield \cite{GW}. As an example we mention that the ring of quotients of a reduced  $n$-subperfect ring is its injective hull.
   
    In view of Sandomierski \cite{Sa}, we are able to obtain interesting results on the homological dimensions of $n$-subperfect rings. We show that in calculating the projective (p.d.), injective (i.d.) and weak (w.d.) dimensions of modules over an $n$-subperfect ring, only the `Goldie part' of the ring counts (see Lemma \ref{H}). 
    
    Let $R$ be an $n$-subperfect ring with minimal prime ideals $P_1, \dots, P_m$. Then $N = P_1 \cap \dots \cap P_m$ is the nilradical of $R$; it is T-nilpotent. By Theorem \ref{Goldie}, $R/N$ is a subdirect product of  {$n$-dimensional}  integral domains $D_j = R/P_j \ (j=1, \dots, m)$. In the next theorem, $D_j$-modules are also regarded as $R$-modules in the natural way.
    
   {\begin{theorem} \label{Sando1} Let $R$ denote an $n$-subperfect ring,  and let $D_j$ be as before. Then an $R$-module $M$ satisfies  {\rm p.d.}$_R M \le k \ (k \ge 0)$ if and only if $\Ext_R^{k+1}(M,X)=0$ for all $D_j$-modules $X$ for each $j=1, \dots, m$.
  \end{theorem}
 
 \begin{proof} See Theorem 5.3  in Sandomierski \cite{Sa}.
  \end{proof}

    \begin{theorem} \label{Sando} Let $R$ be an $n$-subperfect ring,  and $\hat R = R/N$. Then for an $R$-module $M$ we have for any $k \ge 0$: \smallskip
    
    {\rm (a)} {\rm p.d.}$_R M \le k$ if and only if $\Ext_R^{k+1}(M,X)=0$ for all $\hat R $-modules $X$.
        
      {\rm (b)} {\rm i.d.}$_R M \le k$ if and only if $\Ext_R^{k+1}(R/L,M)=0$ for all ideals $L$ containing $N$.
            
        {\rm (c)} {\rm w.d.}$_R M \le k$ if and only if $\Tor^R_{k+1}(R/L,M)=0$ for all ideals $L$ containing $N$.    
 \end{theorem}
 
 \begin{proof} See Theorems 5.2, 3.2,  and 4.2, respectively, in  \cite{Sa}.
  \end{proof}}
  
 Also,  \cite[Proposition 5.4]{Sa} shows that for a flat $R$-module $F$, p.d.$_RF$ can be calculated as  the maximum of the $D_j$-projective dimensions of the flat $D_j$-modules $F \otimes _R D_j$, taken for all $j$.
  
 \smallskip
  
We would like to have information about the finitistic dimensions of $n$-subperfect rings. An estimate is given by \cite[Corollary 1, Section 2]{Sa} which we cite using the same notation as above.

\begin{theorem} \label{dim} For an $n$-subperfect ring $R$ and for the integral domains $D_j$ we have the inequality
$$  \Fdim(R) \le \max_j\{{\rm p.d.}_RD_j + \Fdim(D_j)\}.\qed $$
 \end{theorem}

We recall  (see e.g. Jensen \cite[Remarque, p. 44]{J}) that for a Cohen--Macaulay ring $R$, the finitistic dimension $\Fdim(R)$ is equal either to $d$ or to $d+1$ where $d= \dim R$. For $n$-subperfect rings we do not have such a precise estimate, but we still have some information, see Theorem \ref{findim}. 

In the balance of this section, we will use the notation $\PP_n(R)$ for the class of $R$-modules whose projective dimensions are $\le n$, and $\FF_n(R)$ for the class of modules of weak dimensions $\le n$. We concentrate {on the class $\FF_1(R)$ which is more relevant to subperfectness than the class $\FF_0(R) $ of flat modules}; see e.g. Lemma \ref{H}(v).

  Next, we verify a lemma {(note that $\overline R$-modules may be viewed as $R$-modules).
  
  \begin{lemma} \label{factorsub} Let $R$ be any ring and $\overline R= R/{ rR}$ with $r \in R^\times$ a non-unit.  Then \smallskip
 
 {\rm (i)} if $\overline R$ is subperfect, then $ \FF_1(\overline R) \subseteq \FF_1(R);$
 
 {\rm (ii)} if both $R$ and $\overline R$ are subperfect, then $ \FF_1(R)  \subseteq  \PP_{m}(R)$  for some $m$  implies \ $  \FF_1(\overline R)  \subseteq  \PP_{m-1}(\overline R). $ \end{lemma} 

\begin{proof}   We start observing that if $R$ is a subperfect ring, then a module $H$ satisfies $\Tor_1^R(H,Y) =0$  for all \tf\ $Y$ if and only if $H \in   \FF_1(R) $ (see \cite[Theorem~4.1]{F}); here $Y$ \tf\ means that $\Tor_1^R(R/tR,Y)=0$ for all $t \in R^\times$.  
For any commutative ring $R$, $\Tor_1^R(X,Y)=0$ for all \tf\ $Y$ implies that $X \in \FF_1(R)$ {(but not conversely)}.

  Recall \cite[Chap. VI, Proposition 4.1.1]{CE} which states that if an $R$-module $Y$ satisfies $\Tor_k^R(\overline R, Y)=0$ for all $k >0$, then
$$\Tor_m^R(\overline N ,Y) \cong \Tor_m^{\overline R}(\overline N , Y/rY)  \leqno (3)$$
holds for all $m >0$ and for all $\overline R$-modules $\overline N $. Hypothesis holds if $Y$ is a \tf\ $R$-module: it holds for $k=1$ by definition and for $k>1$ in view of p.d.$_R \overline R =1$. 

First, let $s \in R$ be a divisor of $r$, and choose $\overline N \cong R/sR$. Then the left hand side Tor vanishes for all \tf\ $Y$ and for $m= 1$, so it follows that $Y/rY$ is a \tf\ $\overline R $-module.

(i) {Assuming $\overline R$ is subperfect,} let  $ \overline N \in \FF_1(\overline R)$ and $Y$ a \tf\ $R$-module. The right hand side of (3) vanishes for  $m = 1$, so we can conclude that $\Tor_1^R(\overline N,Y) =0$.  This equality holds for all \tf\ $R$-modules $Y$, whence we obtain $\overline N \in \FF_1(R)$.

(ii)  {Assuming both $R$ and $\overline R$ are subperfect,} let again  $\overline N \in \FF_1(\overline R)$. Part (i) implies that $\overline N \in \FF_1(R)$, so $\overline N \in \PP_m(R)$ by hypothesis. From a well-known Kaplansky formula for projective dimensions {\cite[Proposition 172]{K}} we obtain that $\overline N  \in \PP_{m-1}(\overline R) $, as claimed.
\end{proof}

{\begin{theorem} \label{findim}  If $R$ is an $n$-subperfect ring, then $\Fdim(R) \ge n$.  \end{theorem}

\begin{proof}  According to \cite[Proposition 5.6]{J},  for any ring $R$, $\FF_0(R) \subseteq \PP_{m-1}(R)$ if $m= \Fdim(R)$. Hence we have $\FF_1(R) \subseteq \PP_{m}(R)$. On the other hand, if $R$ is $n$-subperfect, then Lemma \ref{factorsub} is applicable, and by induction it follows that $\FF_1(R) \subseteq \PP_{m}(R)$  {for some $m <n$} would lead to a contradiction that over a subperfect ring of Krull dimension $>0$ modules of weak dimension $\le 1$ are projective. Consequently, $m \ge n$, indeed.
\end{proof} 

That we can have strict inequality in the preceding theorem is shown by the following example.  Let $S$ denote an almost perfect  (i.e., $1$-subperfect) domain; it has finitistic dimension~1. If $R$  is defined as in Example \ref{first ex} as $S \oplus D$ with $D \ne 0$ a \tf\ divisible $S$-module, then p.d.$_RR/D$ is finite and $>1$ {($D$ is flat, but not projective, so p.d.$_RD =1$), whence} Fdim$(R) \ge 2$.

  The following result shows that in Theorem \ref{findim} equality may occur for non-Cohen--Macaulay rings as well. 

  \begin{lemma}  {\rm (i)}  Let $R$ be any ring. Then $\FF_1(R) \subseteq \PP_n(R)$ if and only if $\FF_1(S)  \subseteq  \PP_{n+1}(S)$ holds for {the polynomial ring} $S=R[X]$. 
 
  {\rm (ii)}  If $R$ is a perfect ring, then for {the polynomial ring} $S= R[X_1, \dots, X_n]$ $($which is $n$-subperfect by Theorem~\ref{perfect order}$)$ we have 
  $$\FF_1(S)  \subseteq  \PP_{n}(S), \quad but \qquad \FF_1(S)  \nsubseteq  \PP_{n-1}(S).$$\end{lemma}

\begin{proof} (i) To verify necessity, assume $M$ is a module in $\FF_1(S)$. It is easy to see that then $M \in \FF_1(R)$ as well, thus $M \in \PP_n(R)$ follows by hypothesis. Hence tensoring over $R$ with $R[X]$, we obtain $M[X] \in \PP_n(S)$. It remains to refer to the exact sequence $0 \to M[X] \to M[X] \to M \to 0$ of $S$-modules  to conclude that $M \in \PP_{n+1}(S)$.

Conversely, working toward contradiction, suppose there are an $F \in \FF_1(R)$ and an $H \in$ Mod-$R$ such that $\Ext_R^{n+2}(F,H) \ne 0$. Then also $\Ext_R^{n+2}(F,H[X]) \ne 0. $ Since $\Tor_k^R(F, S)=0$ {for all $k>0$,} we have an isomorphism {(see \cite[Chap. VI, Proposition 4.1.3]{CE}}
$$ \Ext_S^{n+2}(F \otimes_R S,H[X]) \cong \Ext_R^{n+2}(F,H[X]) \ne 0.$$
Since $F \otimes_R S \in \FF_1(S) $,  this is in contradiction to  $\FF_1(S)  \subseteq  \PP_{n+1}(S)$.

(ii)  Noticing that $\FF_1(R)  =  \PP_{0}(R)$ if $R$ is perfect, the claim follows by a simple calculation from (i). 
\end{proof}


 \section{Examples}  \medskip  
 
  {Our final section is devoted to various examples of $n$-subperfect rings. In the first examples we use $n$-subperfect domains to construct $n$-subperfect rings with non-trivial nilradicals. (For examples of non-noetherian $n$-subperfect domains, we refer to Theorem \ref{domain} and Example \ref{D+M} below.)
 
 \begin{example} \label{first ex} Let $S$ denote an $n$-subperfect domain $(n  \ge 1)$ with field of quotients $H$. As a domain, $S$ is trivially a Goldie ring. Let $D$ be a \tf\ divisible $S$-module. Define the ring $R$ as the {\it idealization} of $D$, i.e. $R = S \oplus D$ additively, and multiplication in $R$ is given by the rule 
 $$(s_1,d_1) (s_2,d_2) = (s_1s_2, s_1d_2 + s_2 d_1) \qquad (s_i \in S,\ d_i \in D). $$  
 It is clear that $Q=(H,D)$ is the ring of quotients of $R$, and $N = (0,D)$ is the nilradical  (nilpotent of exponent 2) of both $R$ and $Q$. 
We claim that $R$ is an $n$-subperfect ring. 
 
   {First we observe that an element $r = (s,d) \in R$ is a zero-divisor if and only if $s =0$; this is easily seen by direct calculation using the torsion-freeness of $D$. Hence criterion (iii) in Lemma \ref{H} guarantees that $R$ is a subperfect ring.}    Furthermore, for any $r= (s,d)$, we have $rR = (sS, D)$ (the divisibility of $D$ is relevant).  Therefore, we have an isomorphism $R/rR \cong S/sS$ for every regular $r\in R$ 
 (i.e. for every non-zero $s \in S$). Hence we conclude that $R/rR $ is $(n-1)$-subperfect for every regular $r$ (Corollary \ref{reduction}). By the same Corollary, we obtain the desired conclusion for $R$.  \end{example}
  
   \begin{example} \label{matrix}   As before choose an $n$-subperfect $(n \ge 1$) integral domain $S$.   Let $A$ be any commutative $S$-algebra {that is  \tf\  and divisible as an $S$-module}, and $B$ a \tf\ divisible $S$-module containing $A$. Our ring $R$ is now the ring of upper $3 \times 3$-triangular matrices of the form
  \begin{equation*}
\a = \left[
 \begin{matrix} 
 
 s &  a & b  \\
 0 & s & a  \\
  0 & 0 & s  
  \end{matrix}
   \right  ]  \qquad (s \in S,\ a \in A,\ b \in B).
      \end{equation*}  
It is straightforward to check that $\a \in R$ is a zero-divisor if and only if $s=0$,  and that the principal  ideal $\a R$ equals $sR $  {whenever $s \ne 0$}. Fix  any regular $\a_0 \in R$ (i.e. $0 \ne s_0 \in S$ in the diagonal), and consider the homomorphism $\f: R \to S/s_0S$ given by
$ \a \mapsto s + s_0S \ (\a \in R).$
Then $\Ker \f = \a_0 R = s_0 R$ leads to the isomorphism  $ R / \a_0R \cong S/s_0S$ showing that $ R / \a_0R $ is an $(n-1)$-subperfect ring for every regular $\a_0 \in R$. 
To complete the proof that $R$ is $n$-subperfect, it remains only to  show that $R$ is subperfect. By Lemma \ref{H}(iii) it suffices to observe that the nilradical $N$ of $R$ is nilpotent of degree 3, and every regular coset mod $N$ consists of regular elements of $R$.  \end{example}

   In order to obtain more general examples of similar kind, in the preceding examples we can choose $S$  as a finite direct sum of $n$-subperfect domains. } \smallskip   
             
    Let $R$ be a perfect ring, and let $S = R[X_1,\ldots,X_n]$. By Corollary~\ref{inv} and Theorem~\ref{perfect order}, the ring of invariants $S^G$ of $S$ is $n$-subperfect for each finite group $G$ {acting on $S$} whose order is a unit in $R$.  
 As in the classical case in which $R$ is a field, more examples of $n$-subperfect rings can be obtained from $S$ via Veronese subrings:  a {\it Veronese subring} $T$ of $S$ is an $R$-subalgebra  of $S$ generated by all   monomials of degree $d$ for some fixed $d>0$.  For example, $T= R[X^3, X^2Y, XY^2, Y^3]$ is {a}  Veronese subring of $S = R[X,Y]$ with $d = 3$.  }

     \begin{theorem} \label{Veronese}   Let $R$ be a ring, and $S = R[X_1,\ldots,X_n]$ {a polynomial ring over $R$. A  Veronese subring of $S$ is  $n$-subperfect}  if and only if $R$ is perfect. 
     \end{theorem}
     
     \begin{proof}  Let $T$ be a  Veronese subring of $S$ generated by the monomials of degree $d$. Then $T$ is {an $R$-direct summand of $S$, and $S$ is integral over $T$.} If $R$ is perfect, then $T$ is $n$-subperfect by Theorems~\ref{summand} and~\ref{perfect order}.  Conversely, suppose $T$ is $n$-subperfect. Then $X_1^d,\ldots,X_n^d$ is a maximal regular sequence of $T$, so $T/(X_1^d,\ldots,X_n^d)$ is a perfect ring. As a homomorphic image of this ring, $R$ is perfect. 
     \end{proof}

 Theorem~\ref{poly n-subperfect} shows that if $R$ is perfect, then the ring $R[X_1,\ldots,X_n]$ is $n$-subperfect.  
 As the next example demonstrates, it need not be the case that for a $k$-subperfect ring $R$, $R[X_1,\ldots,X_n]$ is $(n+k)$-subperfect.  
   
   \begin{example}    {Let $F$ be a field, $X, Y$ indeterminates, and $K=F(X)$.}  Then the ring $R = F+ YK[[Y]]$   is an almost perfect domain \cite[Example 3.2]{BS}.  The valuative dimension of $R$, that is, the maximum of the Krull dimensions of the valuation rings of $Q(R)$ that contain $R$, is $2$.  Thus $\dim R[X_1,X_2] = 4$ by \cite[Theorem 6]{Arn}.  Although $R$ is $1$-subperfect, $R[X_1,X_2]$ is not $3$-subperfect, since by Corollary~\ref{Krull} the Krull dimension of a $3$-subperfect ring is~$3$. 
 
   \end{example}
   
    \begin{example} \label{Prufer} A Pr\"ufer domain cannot be $n$-subperfect whenever $n > 1$.   This is because a Pr\"ufer domain $R$ cannot have a regular sequence of length greater than~$1$. Indeed, if $x,y$ is a regular sequence in $R$, then $xR \cap yR = xyR$. If $M$ is a maximal ideal containing $x$ and $y$, then since $R_M$ is a valuation domain, this implies $xR_M = xyR_M$ or $yR_M = xyR_M$, contradicting that neither $x$ nor $y$ is a unit in $R_M$.  {Consequently, an $n$-subperfect Pr\"ufer domain is an almost perfect domain. But for modules over such domains, w.d.$\le 1$ implies p.d.$\le 1$ (see \cite[Theorem 7.1]{FS}), thus any $n$-subperfect Pr\"ufer domain -- if not a field -- must be a Dedekind domain. Dedekind domains are trivially 1-subperfect.}
     \end{example} 
   
    Our next source of examples involves the idealization of a module,   as defined in Example~\ref{first ex}. For an $R$-module $N$, we denote by $R \star N$ the idealization of $N$. 
%
   It is well known that if $R$ is a Cohen--Macaulay ring and $N$ is a finitely generated maximal Cohen--Macaulay module, then $R \star N$ is a Cohen--Macaulay ring.   In Corollary~\ref{idealization cor}, we prove the analogue of this statement for $n$-subperfect rings. This follows from a more general lifting property {of $n$-subperfectness}:

     \begin{theorem} {Let $I$ be  an ideal of the ring  $R$ such that $I^2 = 0$ and $R/I$ is $n$-subperfect for some $n \geq 0$. If every $(R/I)$-regular sequence in $R$ is  also $I$-regular}, then $R$ is $n$-subperfect.
   \end{theorem}

 \begin{proof}   First we show that $R$ is subperfect. 
 {If $N$ is the nilradical of $R$,  then $N/I$ is the nilradical of the $n$-subperfect ring $R/I$, hence T-nilpotent. }Therefore, $N$ as an extension of the nilpotent $I$ by the T-nilpotent $N/I$ is T-nilpotent. Suppose $r + N \ (r \in R)$ is a regular element in $R/N$; then $r +N/I$ is regular in  $(R/I)/(N/I)$, so Lemma \ref{H}(iii) shows that $r +I$ is regular in $R/I$.  Since $r$ is $(R/I)$-regular, $r$ is $I$-regular by assumption. If $r$ is both $(R/I)$-regular and $I$-regular, then it is regular in $R$.  From Lemma \ref{H}(iii) we conclude that $R$ is subperfect.  
 
 We claim next that each $r \in R^\times$ is $(R/I)$-regular. Since $R/I$ is subperfect, there are finitely many prime ideals $P_1,\ldots,P_m$ of $R$ that are minimal over $I$ and whose images in $R/I$ contain every zero-divisor in $R/I$.  {Since $I$ is in the nilradical of $R$,  these primes are also the minimal prime ideals of $R$.}     If $r \in R^\times$, then $r \not \in P_1 \cup \cdots \cup P_m$, so the image of $r$ in $R/I$ is not a zero-divisor. 
 This shows that {the regular elements of  $R$ are $(R/I)$-regular.}

  We prove now using induction that $R$ is $n$-subperfect. If $n  = 0$, then $R/I$ is perfect and hence zero-dimensional. Since $I^2 = 0$, $R$ is zero-dimensional. We have established that $R$ is subperfect, so {from} $R = Q(R)$ we conclude that $R$ is perfect.

     Now let $n >0$, and suppose the lemma has been proved for all $k <n$. We have {already} shown that $R$ is subperfect.   {We claim that  $A := R/rR$ is $(n-1)$-subperfect  for every $r \in R^\times$.} By the induction hypothesis, it suffices to show: \smallskip
     
     (i) $(IA)^2 = 0$,  
     
     (ii) $A/IA$ is $(n-1)$-subperfect, and 
     
     (iii) every $A/IA$-regular sequence in $A$ is $IA$-regular. 
    
     \smallskip 
     It is clear that $(IA)^2 = 0$. 
       To verify (ii), we use the fact already established that  if $r \in R^\times$,  then $r+I$ is {regular} in $R/I$.    Since $R/I$ is $n$-subperfect, Proposition~\ref{factor prop} implies $R/(rR+I) $ is $(n-1)$-subperfect.  {In view of the isomorphism $A/IA \cong R/(rR+I)$,} statement (ii) follows.

     To verify (iii), suppose $a_1,\ldots,a_t$ is an $A/IA$-regular sequence in $A$. 
   {If we write $a_i = r_i + rR$,}  then $r_1,\ldots,r_t$ is an $A/IA$-regular sequence in $R$. Since $r \in R^\times$ and $A/IA \cong R/(rR+I)$, we have that $r,r_1,\ldots,r_n$ is an $R/I$-regular sequence. By assumption, $r,r_1,\ldots,r_t$ is also an $I$-regular sequence, so $r_1,\ldots,r_t$ is an $I/rI$-regular sequence. As established, every regular element of $R$ is a regular element in $R/I$. Thus $I \cap rR = rI$, and it follows that 
      $IA = (I + rR)/rR \cong I/(I \cap rR) = I/rI$. Since $r_1,\ldots,r_t$ is an $(I/rI)$-regular sequence in $R$, we conclude that $a_1,\ldots, a_t$ is an $IA$-regular sequence in $A$.  Thus every $A/IA$-regular sequence in $A$ is $IA$-regular.
      
      Having verified (i), (ii) and (iii), we conclude from the induction hypothesis that $A=R/rR$ is $(n-1)$-subperfect. Since $R$ is subperfect and $R/rR$ is $(n-1)$-subperfect for each $r \in R^\times$, {Corollary~\ref{reduction}} implies $R$ is $n$-subperfect. 
   \end{proof}

   \begin{corollary} \label{idealization cor} Let $R$ be an $n$-subperfect ring, and let $N$ be an $R$-module such that every regular sequence in $R$ extends to a regular sequence on $N$. Then $R \star N$ is an $n$-subperfect ring. \qed
   \end{corollary}
   
   \begin{example} \label{idealization ex} Corollary~\ref{idealization cor} implies that if $R$ is a local Cohen--Macaulay ring,  and if $N$ is a balanced big Cohen--Macaulay $R$-module, then $R \star N$ is $n$-subperfect {for $n = \dim R$}. 
      {Choosing $N$ to be an infinite rank free $R$-module, we obtain   a non-noetherian $n$-subperfect ring $R \star N$.} \smallskip
   
   More interesting choices are possible for $N$. For example if $R$ is an excellent local Cohen--Macaulay domain of positive characteristic and $R^+$ is the integral closure of $R$  in the algebraic closure of the quotient field of $R$, then $R \star R^+$ is a non-noetherian $n$-subperfect ring, since $R^+$ is a balanced big Cohen--Macaulay module that is not  finitely generated \cite[Theorem 1.1]{HH}. 
   \end{example} 
   
   \begin{example} \label{infpoly} {Let $R$ be an $n$-subperfect ring and $\{X_i:i \in I\}$ be a collection of indeterminates for $R$. Let
   $$S=R[X_i:i \in I]/(X_i:i \in I)^2.$$  The maximal ideal $N = (X_i:i \in I)/(X_i:i \in I)^2$ of $S$ is nilpotent of index $2$ and  is a free $R$-module with basis the images of the $X_i$ in $N$.  As   
    $S \cong R \star N$, the ring $S$ is a special case of the construction in Example~\ref{idealization ex}; therefore, $S$ is $n$-subperfect. If the index set $I$ is infinite, then $S$ is  not noetherian. }
   \end{example}

   So far, our non-noetherian examples, at least for $n>1$,  have involved $n$-subperfect rings with zero-divisors. 
   Our next source of examples produces non-noetherian $n$-subperfect domains, albeit in a non-transparent way.

    \begin{theorem} \label{part}
    Let $S$ be a local  Cohen--Macaulay  domain such that $Q(S)$ is separably generated, and has positive characteristic and uncountable transcendence degree over its prime subfield.  If $n:=\dim S \geq 1$, then 
    there exists a non-noetherian $n$-subperfect subring $R$ of $S$ such that $Q(R) = Q(S)$ and  $S$ is integral over $R$.  
     \end{theorem} 
     
     \begin{proof}  Let $N$ be a free $S$-module of infinite rank. Applying \cite[Theorem 3.5]{OCounter} to $S$ and $N$, we obtain  
    a subring $R$ of $S$ such that
    $R$ is ``strongly twisted by $N$''. We omit the definition of this notion here, but we use the fact that by \cite[Theorems~4.1 and~4.6]{OCounter} this implies 
    
    (i) 
    there is a subring $A$ of $R$ such that $S/A$ is a torsion-free divisible $A$-module and  $I \cap A \ne 0$ for each ideal $I$ of $S$;
    
    (ii) 
     $R$ has the same quotient field as $S$ and $S$ is an integral extension of $R$; and 
     
     (iii)  there is  a faithfully flat ring embedding $f:R \rightarrow S \star N$ such that for each $0 \ne a \in A$,  the induced map $f_a:R/aR \rightarrow (S\star N)/a(S \star N)$ is an isomorphism. 
     
      {We show that for each nonempty regular sequence $x_1,\ldots, x_t$ in  $R$, the ring $R/(x_1,\ldots,x_t)R$ is subperfect.}  Since $f$  is faithfully flat,  $f(x_1),\ldots,f(x_t)$ is a regular sequence in $T:=S \star N$.  By Corollary~\ref{idealization cor}, $T$ is an $n$-subperfect ring. Thus $f(x_1),\ldots,f(x_t)$ is a subperfect sequence in $T$.  Since for each $0 \ne a \in A$, the map $f_a$ is an isomorphism, we have $T = f(R) + f(a)T$.  By (i) and (ii), the fact that $S/R$ is a torsion $R$-module implies there is $0 \ne a \in (x_1,\ldots,x_t)R \cap A$.  Hence
      $$T = f(R) + (f(x_1),\ldots,f(x_t))T.$$ 
   Moreover, since $f$ is faithfully flat, we have 
   $$\left(f(x_1),\ldots,f(x_t)\right)T \cap f(R) = (f(x_1),\ldots,f(x_t))f(R).$$  Therefore, 
     \begin{eqnarray*} T/(f(x_1),\ldots,f(x_t))T & = & (f(R) + (f(x_1),\ldots,f(x_t))T)/(f(x_1),\ldots,f(x_t))T \\
     & \cong & f(R)/((f(x_1),\ldots,f(x_t))T \cap f(R))  \\
     & = & f(R)/(f(x_1),\ldots,f(x_t))f(R)  \\
     & \cong & R/(x_1,\ldots,x_t)R.
     \end{eqnarray*}  
     Consequently, since $f(x_1),\ldots,f(x_t)$ is a subperfect sequence in $T$, it follows that $x_1,\ldots,x_t$ is a subperfect sequence in $R$.  This proves that every regular sequence in $R$ is subperfect.

   Finally,  since $S$ is integral over $R$ and $S$ is local, $R$ is also local and has the same Krull dimension as $S$. By Corollary~\ref{Krull}, $n = \dim S = \dim R$.   {Taking into account that} every regular sequence in $R$ is subperfect,  Corollary~\ref{Krull2}  implies that $R$ is $n$-subperfect.  
   By \cite[Theorem 5.2]{OCounter}, the fact that $N$ is a free $S$-module of infinite rank implies $R$ is not noetherian.     
    \end{proof} 
    
    \begin{example} \label{domain} Let $p$ be a prime number, and let ${\mathbb F}_p$ denote the field with $p$ elements. Suppose $k$ is a purely transcendental extension of ${\mathbb F}_p$ with uncountable transcendence degree.   Then 
     $S = k[X_1,\ldots, X_n]_{(X_1,\ldots,X_n)}$  is a local $n$-subperfect domain (in fact, a Cohen--Macaulay ring) meeting the requirements of Theorem~\ref{part}. {Thus $S$ contains a non-noetherian $n$-subperfect subring $R$}  having the same quotient field as $S$. 
     \end{example}

    Our final source of examples involves local Cohen--Macaulay rings that have a coefficient field. The next theorem shows that restriction to a smaller coefficient field can produce  examples of non-noetherian $n$-subperfect rings. 
    
     \begin{theorem} \label{D+M} Let  $S$ be a local  Cohen--Macaulay ring containing a field $F$ such that $S = F+M$, where $M$ is the maximal ideal  of $S$. For each subfield $k$ of $F$, {the local ring $R = k+M$ is $n$-subperfect for $n = \dim S$.} The ring $R$ is noetherian if and only if $F/k$ is a finite extension.
  \end{theorem}
 
  \begin{proof} Let $k$ be a subfield of $F$, and let $R = k + M$. Then $R$ is  a local ring with maximal ideal $M$. It is clear that every prime ideal of $S$ is a prime ideal of $R$. 
To {verify the converse,} let  $P$ be a non-maximal prime ideal of $R$. To show  that $P$ is in fact an ideal of $S$, let $s \in S$. Then $sP \subseteq sM \subseteq R$, and also, $(sP)M = P(sM) \subseteq P$ {because} $sM \subseteq R$.  Since $M \not \subseteq P$, we conclude that $sP \subseteq P$, which proves that $P$ is an ideal of $S$. To see that {$P$ is  prime in $S$,} let $x,y \in S$ with $xy \in P$. If  one of $x$ or $y$ is a unit in $S$, then the other is in $P$. Otherwise, if neither $x$ nor $y$ are units, then {necessarily} $x,y \in M \subseteq R$, and since $P$ is a prime ideal of $R$, one of $x,y$ is in $P$.  Thus $P$ is a prime ideal of $S$, and this shows that 
 the prime ideals of $R$ are precisely {those} of $S$. 
 
  We show now that $R$ is $n$-subperfect, where $n = \dim S$. 
  By \cite[Lemma 1.1.4, p.~5]{FHP},  $Q(R) = Q(S)$, so $R$ is a subperfect ring, since the total quotient ring $Q(S)$ of the Cohen--Macaulay ring $S$  is artinian.
  Let  $x_1,\ldots,x_t $ be a regular sequence in $R$, and $I = (x_1,\ldots,x_t)R$.   We claim that $R/I$ is a subperfect ring.   The height of $I$ in $R$ is at least $t$, and since $R$ and $S$ share the same prime ideals,  the height of $IS$ is also at least $t$.  Krull's height theorem implies then  that the height of the $t$-generated ideal $IS$ is  $t$. 
 Since $S$ is a Cohen--Macaulay ring, the ideal  $IS$ is  unmixed.   We use this to show next that $Q(R/I)$ is zero-dimensional.  

To this end, we {prove} that every zero-divisor of $R/I$ is contained in a minimal prime ideal of $R$.  Let  $x,y \in R$ such that $xy \in I$ and $y \not \in I$.
 Suppose  by way of contradiction that  $x$ is not {contained} in any minimal prime ideal of $I$. Since $I$ and $IS$ share the same minimal primes,  the image of $x$ in $S/IS$ {does not belong to} any minimal prime ideal of $S/IS$. However, $IS$ is unmixed, so necessarily $y \in IS$.  Therefore, using the fact that $S = F+M$, we can write 
 \begin{center}$y = \alpha_1x_1 + \cdots + \alpha_t x_t + z$ \qquad for  $\alpha_1,\ldots,\alpha_t \in F$\ and\ $z \in (x_1,\ldots,x_t)M$.
   \end{center}
  Similarly, since $xy \in I$ and $R = k + M$, we have
  \begin{center}  $xy  = \beta_1x_1 + \cdots + \beta_t x_t + w$ \qquad for  $\beta_1,\ldots,\beta_t \in k$\  and \ $w \in (x_1,\dots,x_t)M$.   \end{center}
    Let $i$ be the largest {index}   such that at least one of $\alpha_i,\beta_i$ is not $0$.
    Using the preceding  expressions for $y$ and $xy$, we obtain
       $$\beta_1x_1 + \cdots + \beta_i x_i + w = \alpha_1xx_1 + \cdots + \alpha_i xx_i + xz.$$
Therefore,   
    $$  (\beta_i - \alpha_ix)x_i \in (x_1,\ldots,x_{i-1})R.$$
    Since $\beta_i - \alpha_i x \in k+M = R$ and $x_1,\ldots,x_i$ is a regular sequence in $R$,
    we have  $\beta_i -\alpha_i x \in (x_1,\ldots,x_{i-1})R$.   The fact that $x$ is a non-unit in $R$ implies $\beta_i \in M$, so $\beta_i = 0$ and hence, by the choice of $i$, $\alpha_i \ne 0$. 
    Since the prime ideals of $S$ are the same as the prime ideals of $R$, $\sqrt{(x_1,\ldots,x_{i-1})R}$ is an ideal of $S$.  Also, $\alpha_i$ is a unit in $S$ and  $\alpha_i x \in (x_1,\ldots,x_{i-1})R$, so $x \in \sqrt{(x_1,\ldots,x_{i-1})R} \subseteq \sqrt{I}$. However,  $x$ was chosen not to be contained in  any minimal prime ideal of $I$. This contradiction implies that $x$ must be in some minimal prime ideal of $I$, {establishing} that $Q(R/I)$ is a zero-dimensional ring. 
Since $I$ and $IS$ share the same minimal prime ideals, {$I$ has only finitely many minimal primes, so $Q(R/I)$ is also semilocal.}

 It remains to show that the nilradical of $R/I$ is T-nilpotent, and to prove this, it suffices to show that some power of $\sqrt{I}$ is contained in $I$.  
    Since $\sqrt{I}$ is  a finitely generated ideal of the noetherian ring $S$ and $ \sqrt{I} = \sqrt{IM}$, with $IM$ an ideal of $S$, there is $t>0$ such that $(\sqrt{I})^t \subseteq IM \subseteq I$. Therefore, $R/I$ is subperfect, which completes the proof that every regular sequence in $R$ is subperfect. Since $R$ and $S$ share the same maximal ideal, Corollary~\ref{localization cor}  implies $R$ is $n$-subperfect for $n = \dim S$. {Finally,} it is straightforward to check that $R$ is noetherian if and only if $F/k$ is a finite field extension; see \cite[Proposition 1.1.7, p.~7]{FHP}.   \end{proof}

    \begin{example} Let $S 
   = F[[X_1,\ldots,X_n]]/I$, where $F$ is a field, $X_1,\ldots,X_n$ are indeterminates for $F$, and $I$ is an ideal such that $S$ is Cohen--Macaulay. Theorem~\ref{D+M} implies that  for each subfield $k$ of $F$, 
   \begin{center}$R = \{f + I \in S:f \in F[[X_1,\ldots,X_n]]$ and $ f(0,\ldots,0) \in k\}$
   \end{center} is an $n$-subperfect ring. 
   \end{example}


\end{document}